\documentclass[12pt]{iopart}
\expandafter\let\csname equation*\endcsname\relax
\expandafter\let\csname endequation*\endcsname\relax

\usepackage{amsmath,amsfonts,amssymb,amsthm}
\usepackage{graphicx}
\usepackage{color}
\usepackage{hyperref}

\def\softd{{\leavevmode\setbox1=\hbox{d}%
\hbox to 1.05\wd1{d\kern-0.4ex{\char039}\hss}}}
\def\softt{{\leavevmode\setbox1=\hbox{t}%
\hbox to \wd1{t\kern-0.6ex{\char039}\hss}}}

\newcommand{\R}{\mathbb{R}}

\newtheorem{theorem}{Theorem}[section]
\newtheorem{lemma}{Lemma}[section]
\newtheorem{proposition}{Proposition}[section]
\newtheorem{remark}{Remark}[section]

\definecolor{Myred}{cmyk}{0.0,1.0,1.0,0.00}
\renewcommand{\theequation}{\arabic{section}.\arabic{equation}}

\begin{document}

\title[A magnetic version of the Smilansky-Solomyak  model]
{A magnetic version of the Smilansky-Solomyak  model}

\author{Diana Barseghyan}
\address{Department of Mathematics, University of Ostrava, 30. dubna 22, 70103 Ostrava, Czech Republic}
\address{Nuclear Physics Institute, Academy of Sciences of the Czech Republic,
Hlavn\'{i} 130, 25068 \v{R}e\v{z} near Prague, Czech Republic}
\ead{dianabar@ujf.cas.cz, diana.barseghyan@osu.cz}

\author{Pavel Exner}
\address{Nuclear Physics Institute, Academy of Sciences of the Czech Republic,
Hlavn\'{i} 130, 25068 \v{R}e\v{z} near Prague, Czech Republic}
\address{Doppler Institute, Czech Technical University, B\v{r}ehov\'{a} 7, 11519 Prague, Czech Republic}
\ead{exner@ujf.cas.cz}

\begin{abstract}
We analyze spectral properties of two mutually related families of
magnetic Schr\"{o}dinger operators, $H_{\mathrm{Sm}}(A)=(i \nabla
+A)^2+\omega^2 y^2+\lambda y \delta(x)$ and $H(A)=(i \nabla
+A)^2+\omega^2 y^2+ \lambda y^2 V(x y)$ in $L^2(\R^2)$, with the
parameters $\omega>0$ and $\lambda<0$, where $A$ is a vector
potential corresponding to a homogeneous magnetic field
perpendicular to the plane and $V$ is a regular nonnegative and
compactly supported potential. We show that the spectral properties
of the operators depend crucially on the one-dimensional
Schr\"{o}dinger operators $L= -\frac{\mathrm{d}^2}{\mathrm{d}x^2}
+\omega^2 +\lambda \delta (x)$ and $L (V)= -
\frac{\mathrm{d}^2}{\mathrm{d}x^2} +\omega^2 +\lambda V(x)$,
respectively. Depending on whether the operators $L$ and $L(V)$ are
positive or not, the spectrum of $H_{\mathrm{Sm}}(A)$ and $H(V)$
exhibits a sharp transition.

\vspace{2pc} \noindent{\it Keywords}: Discrete spectrum, essential
spectrum, Smilansky-Solomyak model, spectral transition, homogeneous
magnetic field
\end{abstract}

\submitto{J. Phys. A.: Math. Theor.}

\maketitle

\section{Introduction} \label{s: intro}
\setcounter{equation}{0}

Irreversible dynamics of quantum systems is usually described
through a coupling of the object, regarded as an open or unstable
system, to another one that plays the role of a `heat bath'. The
latter is usually supposed to be a `large' system having an infinite
numbers of degrees of freedom and the Hamiltonian with a continuous
spectrum, moreover, the presence (or absence) of irreversible modes
is determined by the energies involved rather than the coupling
strength between the object and the bath. While this is all true in
many cases, it need not be true in general. This was demonstrated by
Uzy Smilansky using a simple model \cite{Sm04} which was
subsequently analyzed in detail and generalized by Mikhail Solomyak
and coauthors \cite{So04, S04, ES05a, ES05b, So06a, So06b, NS06,
RS07}, see also \cite{Gu11} and \cite{ELT17}.

In the simplest case the model is described by the Hamiltonian
\begin{equation}
\label{HSmil} H_\mathrm{Sm}=-\frac{\partial^2}{\partial x^2}
+\frac12\left( -\frac{\partial^2}{\partial y^2}+y^2 \right) +\lambda
y\delta(x)
\end{equation}
in $L^2(\mathbb{R}^2)$ with the natural domain and exhibits a
transition between two types of spectral behavior: for
$|\lambda|\le\sqrt{2}$ the operator \eqref{HSmil} is bounded from
below, while for $|\lambda|>\sqrt{2}$ its spectrum fills the real
line \cite{So04}. The factor $\frac12$ is not important and can be
changed by a scaling of one of the variables. If we replace it by
one, for instance, the critical value of the coupling constant will
be $\lambda=2$. The transition between the two regimes can be
interpreted also dynamically \cite{Gu11}: in the supercritical
regime the $y$-dependent binding energy of $\delta$ interaction wins
over the oscillator potential and the wave packet can escape to
infinity along the singular channel.

While mathematically we deal with the same object, from the physical
point of view one can interpret it in two different ways. In the
original Smilansky paper \cite{Sm04} it was meant as a system of two
one dimensional components, a particle motion on a line to which a
heat bath consisting of a single harmonic oscillator is coupled in a
coordinate dependent way. In the generalizations mentioned above the
line was replaced by other simple configuration spaces, a loop (in
other words, segment with periodic boundary conditions) or a graph,
and the bath could be anharmonic or multidimensional (but still with
a finite number of degrees of freedom).

Another point of view, which can be associated with the work of
Solomyak and coauthors, is to associate the Hamiltonian
\eqref{HSmil} with a two-dimensional system in which the particle
moves in the potential which is the sum of the oscillator `channel'
and the singular component with the position-dependent coupling
strength. Viewed from this angle, the system brings to mind motion
in a potential with channels which are below unbounded and narrowing
towards infinity. In this situation one may also observe a jump
transition from a below bounded to below unbounded spectrum as was
first noted in \cite{Zn98}, a class of models of this type was
analyzed recently in \cite{EB12, BEKT16}. The analogy becomes even
more convincing when we recall that the model \eqref{HSmil} has a
`regular' analogue \cite{BE14, BE17} in which the $\delta$
interaction is replaced by a non-singular potential properly scaled.
In this case, of course, the `first', two subsystem, interpretation
is lost unless we try to interpret the potential as a sort of
nonlocal position-dependent subsystem coupling.

The main question we want to address in the present paper is what
will happen with the model in its two-dimensional version when the
particle it describes is charged and exposed to a homogeneous
magnetic field perpendicular to the plane, in other words, what are
the properties of the operator
\begin{equation}\label{Hmagn1}
H_{\mathrm{Sm}}(A)=(i \nabla +A)^2+\omega^2 y^2+\lambda y \delta(x)
 \end{equation}
with $\lambda\in\mathbb{R}$ and $\omega>0$, where the vector
potential $A$ corresponds to the indicated magnetic field of
intensity $B>0$. We may consider non-positive $\lambda$ only; as in
the nonmagnetic case this is due to mirror symmetry but the argument
is a bit trickier. The magnetic field changes direction when
observed in a mirror, however, switching the sign of both the
variables we return to the original $A$ and at the same time the
last term on the right-hand side of \eqref{Hmagn1} changes sign. It
is clear that now the two-subsystem interpretation is ultimately
lost, therefore it is appropriate to speak of \eqref{Hmagn1} as of
the Hamiltonian of the \emph{magnetic Smilansky-Solomyak model.}

The dynamics of the model combines the influence of several forces
and its properties are not \emph{a priori} obvious. For $\lambda=0$
the spectrum is absolutely continuous and the particle moves along
the parabolic channel provided its energy is larger than
$\sqrt{\omega^2+B^2}$, and moreover, the transport is stable against
localized perturbations \cite{EK00}. If both $\omega$ and $\lambda$
vanish, operator \eqref{Hmagn1} is the Landau Hamiltonian the
spectrum of which is known to be pure point, consisting of the
Landau levels $(2n+1)B,\: n=0,1,2,\dots$. Was the singular term
position independent, just $\lambda \delta(x)$, it would make the
spectrum absolutely continuous corresponding to transport in the $y$
direction as one could check in a way similar to the Iwatsuka model
\cite[Sec.~6.5]{CFKS87}, \cite[Sec.~7.2.3]{EK15}, or to magnetic
transport along a barrier \cite{FGW00}. The operator \eqref{Hmagn1}
with $\omega=0$ has not been analyzed to the best of our knowledge,
but one can expect that it will exhibit some transport properties
again; at least we will show, as a byproduct of our results here,
that its spectrum covers the whole real axis whenever $\lambda\ne
0$.

The oscillator potential, however, acts against a transport in the
$y$ direction. We are going to show that the resulting behavior is
determined by the balance of the two forces, in a way to a large
degree similar to the nonmagnetic case, $A=0$. To be specific, we
introduce the  comparison operator,
\begin{equation}\label{comparison1}
L=-\frac{\mathrm{d}^2}{\mathrm{d}x^2}+\omega^2+\lambda\delta(x)
\end{equation}
on $L^2(\R)$ with the usual domain \cite{AGHH05}; our goal is to
establish a correspondence between the spectral regime of
$H_{\mathrm{Sm}}(A)$ and the positivity of the operator
\eqref{comparison1}. We are going to show that the spectrum is
bounded from below provided $\inf\sigma(L)>0$ --- we speak here
about the \emph{subcritical} case --- and it has a purely discrete
character below $\sqrt{\omega^2+B^2}$ being nonempty whenever
$\lambda\ne 0$. In the \emph{critical} case, $\inf\sigma(L)=0$, the
operator $H_{\mathrm{Sm}}(A)$ remains positive but its spectrum is
purely essential and equal to $[0,\infty)$. Finally, if one passes
to the \emph{supercritical} regime, $\inf\sigma(L)<0$, an abrupt
transition occurs and the spectrum fills now the whole real line.

In a similar way, the `regular' version of the model \cite{BE14,
BE17} mentioned above has also its magnetic counterpart. The
dynamics is this case described by the Hamiltonian
\begin{equation}\label{Hmagn2}
H(A)=(i \nabla +A)^2+\omega^2 y^2+ \lambda y^2 V(x y)\,,
\end{equation}
where $V$ is a nonnegative, sufficiently smooth function with
$\mathrm{supp}(V)\subset[-s_0, s_0]$ for some $s_0>0$, furthermore,
$\lambda\le 0$ and  $\omega>0$, and the magnetic potential $A$
corresponds as before to a homogeneous magnetic field of the
intensity $B>0$. Note that the analogy is not complete because both
parts of the scalar potential are mirror symmetric with respect to
the $x$ axis, however, the effect which we are interested in depends
on the presence of an attractive interaction in the $y$ direction,
irrespective whether is one- or two-sided. Spectral properties of
the operator \eqref{Hmagn2} will be the topic of the second part of
the paper. The abrupt spectral transition occurs here again. The
comparison operator  will be
\begin{equation}\label{comparison2}
L(V)=-\frac{\mathrm{d}^2}{\mathrm{d}x^2}+\omega^2+\lambda V(x)
\end{equation}
on $L^2(\mathbb{R})$ with the domain $\mathcal{H}^2(\R)$, and its
spectral threshold will be shown to be decisive: $H(A)$ will be
bounded from below provided if $L(V)$ is nonnegative, and its
spectrum will fill the whole real line in the opposite case.

\section{Spectrum of $H_{\mathrm{Sm}}(A)$}
\label{s:Smil} \setcounter{equation}{0}

Before coming to our proper subject we note that in order to
interpret $H_{\mathrm{Sm}}(A)$ as a quantum mechanical Hamiltonian,
one has check its self-adjointness. In the subcritical case, when
$L$ is strictly positive, we can consider first the domain
$\mathcal{D}_0$ consisting of the family of functions $v$ twice
differentiable except at the $y$ axis, $x=0$, continuous there and
satisfying the matching conditions $\frac{\partial v}{\partial
x}(0+, y)-\frac{\partial v}{\partial x}(0-, y)=\lambda y v(0, y)$
and to identify $H_{\mathrm{Sm}}(A)$ with the \emph{Friedrichs
extension} of such an operator. In other words, we will consider the
quadratic form
\begin{equation} \label{Smform}
Q (H_{\mathrm{Sm}} (A))[u]=\int_{\mathbb{R}^2}\bigg[ \left|i
\frac{\partial u}{\partial x}-B y u\right|^2+\left|\frac{\partial
u}{\partial y}\right|^2+\omega^2 y^2 |u|^2
\bigg]\,\mathrm{d}x\,\mathrm{d}y+\lambda \int_{\mathbb{R}}y |u(0,
y)|^2\,\mathrm{d}y
\end{equation}
and demonstrate that it is closed on
$\mathcal{D}:=\mathcal{H}^1(\mathbb{R}^2)\cap
\left\{\int_{\mathbb{R}^2}y^2 |u(x,y)|^2\,\mathrm{d}x\,
\mathrm{d}y<\infty\right\}$ and bounded from below provided $\inf
\sigma(L)>0$, and therefore associated with a unique self-adjoint
operator.

The approach based on quadratic forms fails, of course, if we cannot
ensure that the operator is bounded from below. To make things
simple, we can bypass this trouble by noting that
$H_{\mathrm{Sm}}(A)$ is essentially self-adjoint on $\mathcal{D}_0$. Indeed,
it is easy to check that such a operator is densely defined and
symmetric. To ensure that its deficiency indices coincide, it is by
\cite[Thm.~2.8]{GMNT16} sufficient to check that its commutes with a
\emph{conjugation}, i.e. an antilinear map $L^2(\R^2) \to L^2(\R^2)$
which is norm preserving and idempotent; choosing $\mathcal{C}:\:
(\mathcal{C}v)(x,y) = \overline{u(-x,y)}$ we get the claim. Then we
know that $H_{\mathrm{Sm}}(A)\upharpoonright D_0$ has self-adjoint
extension and will show that the indicated spectral properties hold
for \emph{any} such extension.

\subsection{The subcritical case}
\label{ss:Smil-subcrit-sa}

As we have indicated, to establish the self-adjointness of
$H_{\mathrm{Sm}}(A)$ one has to check that the form \eqref{Smform}
is bounded from below and closed if $\inf\sigma(L)= \omega^2-\frac14
\lambda^2>0$. First we will show that the operator is in fact
positive even if the last inequality is not sharp.
\begin{proposition} \label{subcr-posit}
Let $\lambda\ge-2\omega$, then $H_{\mathrm{Sm}}(A)\upharpoonright
\mathcal{D}_0 \ge 0$.
\end{proposition}
\begin{proof}
For every $u\in \mathcal{D}_0$ the form \eqref{Smform} can be
estimated form below by neglecting the `transverse' contribution to
the kinetic energy
$$
Q( H_{\mathrm{Sm}} (A))[u] \ge \int_{\mathbb{R}^2} \bigg[\left|i
\frac{\partial u}{\partial x}-B y u\right|^2+\omega^2 y^2
|u|^2\bigg]\,\mathrm{d}x\,\mathrm{d}y\\+\lambda \int_{\mathbb{R}}y
|u(0, y)|^2\,\mathrm{d}y\,.
$$
For any fixed $y$ the form $u(\cdot,y)\mapsto
\int_{\mathbb{R}}\left|i \frac{\partial u}{\partial x}-B y
u\right|^2+\omega^2 y^2 |u|^2\,\mathrm{d}x+\lambda y |u(0, y)|^2$
corresponds to the essentially self-adjoint operator
$$
\left(i \frac{\mathrm{d}}{\mathrm{d}x}-B y\right)^2+\omega^2
y^2+\lambda y \delta(x)
$$
whose closure has $\mathcal{H}^1(\mathbb{R})$ as its form domain.
This operator is unitarily equivalent to $y^2 L$ which is positive
by assumption if $y>0$, and to $y^2\widetilde{L}\ge0$ if $y<0$,
where
\begin{equation}\label{widetildeL}
\widetilde{L}:=-\frac{d^2}{d x^2}+\omega^2-\lambda \delta(x)\,;
\end{equation}
this establishes the sought claim.
\end{proof}

\begin{proposition}
The form \eqref{Smform} is closed if $\lambda>-2\omega$.
\end{proposition}
\begin{proof}
Let $\{u_n\}_{n=1}^\infty\subset\mathcal{D}$ be a sequence
converging to some $u\in L^2(\R^2)$ and satisfying
\begin{equation} \label{Cauchy}
Q(H_{\mathrm{Sm}})(A)[u_n-u_m]\to 0
\end{equation}
as $m,n\to\infty$. By the assumption one can choose $\alpha\in
\big(\frac{|\lambda|}{2\omega},1\big)$ and rewrite the form value in
question as
\begin{eqnarray}
\nonumber \hspace{-4em} Q( H_{\mathrm{Sm}}(A))[u_n-u_m] =
(1-\alpha)\biggl(\int_{\mathbb{R}^2}\left|(i
\nabla+A)(u_n-u_m)\right|^2\,\mathrm{d}x\,\mathrm{d}y\\\label{closedness}
\hspace{-2.5em} +\omega^2\int_{\mathbb{R}^2} y^2
|u_n-u_m|^2\,\mathrm{d}x\,\mathrm{d}y\biggr)+\alpha\biggl(\int_{\mathbb{R}^2}\left|(i
\nabla+A)(u_n-u_m)\right|^2\,\mathrm{d}x\,\mathrm{d}y\\\nonumber
\hspace{-2.5em} +\omega^2 \int_{\mathbb{R}^2}y^2
|u_n-u_m|^2\,\mathrm{d}x\,\mathrm{d}y+\frac{\lambda}{\alpha}
\int_{\mathbb{R}}y |u_n(0, y)-u_m(0, y)|^2\,\mathrm{d}x\biggr)\,.
\end{eqnarray}
In the same way as in the proof of the previous proposition one can
check that the second summand on the right-hand side of
(\ref{closedness}) is nonnegative, hence neglecting it we estimate
$Q( H_{\mathrm{Sm}}(A))[u_n-u_m]$ from below by the first summand
which in view of \eqref{Cauchy} tends to zero as $m,n\to\infty$.
Since $\alpha\ne 1$ by construction, this means that the sequence
$\{Q_0(A)[u_n]\}$ is Cauchy, where $Q_0(A)$ is the `unperturbed'
form
$$
u\mapsto Q_0(A)[u]:=\int_{\mathbb{R}^2} \big[\left|(i
\nabla+A)u\right|^2+\omega^2 y^2
|u|^2\big]\,\mathrm{d}x\,\mathrm{d}y
$$
defined on $\mathcal{D}$. It is not difficult to verify that
$Q_0(A)$ is closed and this in turn implies that the limit function
$u$ belongs to $\mathcal{D}$ and
$$
\int_{\mathbb{R}^2}\big[\left|(i \nabla+A)(u_n-u)\right|^2+\omega^2
y^2 |u_n-u|^2\big]\,\mathrm{d}x\,\mathrm{d}y\to0\quad\text{as}\quad
n\to\infty\,,
$$
cf.~\cite[Problem~VIII.15]{RS80}. It remains to check that
$\int_{\mathbb{R}}y |u_n(0, y)-u(0, y)|^2\,\mathrm{d}y \to 0$ holds
as well. Using a couple of simple estimates,
\begin{eqnarray*}
\int_{\mathbb{R}}|y||v(0, y)|^2\,\mathrm{d}y \le \frac{2}{\omega}
\int_{\mathbb{R}^2}\left(\left|\frac{\partial v}{\partial
x}\right|^2+\omega^2y^2 |v(x, y)|^2\right)\,\mathrm{d}x\,\mathrm{d}y
\\ \quad \le \frac{1}{2\omega}
\int_{\mathbb{R}^2}\left(\left|\frac{\partial v}{\partial
x}\right|^2+\left|\frac{\partial v}{\partial y}+i B x v\right|^2(x,
y)+ \omega^2y^2 |v(x, y)|^2\right)\,\mathrm{d}x\,\mathrm{d}y \\
\quad = \frac{1}{2\omega}\, Q_0(a)[v]\,,
\end{eqnarray*}
and inserting $v=u_n-u$, we conclude the proof.
\end{proof}

This guarantees that in the subcritical case there is a unique
self-adjoint operator $H_{\mathrm{Sm}}(A)$ associated with the form
\eqref{Smform}.

\subsection{The essential spectrum}
\label{ss:Smil-subcr-ess}

In fact will first show that the essential spectrum is nonempty
independently of $\lambda$.

\begin{theorem} \label{th:subcritess}
$\sigma_{\mathrm{ess}} (H_{\mathrm{Sm}}(A))\supset
[\sqrt{\omega^2+B^2}, \infty)$.
\end{theorem}
\begin{proof}
It is sufficient to construct a Weyl sequence for any number
$\mu>\sqrt{\omega^2+B^2}$. To this aim, we fix first a positive
number $\varepsilon$ and construct a function $\phi$  such that
\begin{equation}\label{Weyl ineq.}
\|H_{\mathrm{Sm}} (A) \phi-\mu \phi\|_{L^2
(\mathbb{R}^2)}<\varepsilon \|\phi\|\,.
\end{equation}
We employ the functions
\begin{eqnarray}\label{Functions}\lefteqn{
\varphi_{k, \alpha, m}(x, y):=}\\ \nonumber &&\frac{1}{\sqrt{2\pi\,
\mathrm{vol}(E)}}\left(\int_E g\left(y-\frac{\xi
B}{\omega^2+B^2}\right)\, e^{i \xi (x-\alpha
k)}\,\mathrm{d}\xi\right)\,\eta\left(\frac{x}{k}\right)\,\chi\left(\frac{y}{k}\right)\,,
\end{eqnarray}
where $g$ is the normalized eigenfunction associated with the
principal eigenvalue of the harmonic oscillator modified by the
presence of the magnetic field, $h_{\mathrm{osc}}=
-\frac{\mathrm{d}^2}{\mathrm{d}y^2}+\left(\omega^2+B^2\right) y^2$
on $L^2(\mathbb{R})$, the functions $\eta\in C_0^\infty (1, m)$\,,
$\chi\in C_0^\infty(-1, 1)$ are supposed to satisfy the following
requirements,
$$
\eta(z)\ge\frac{1}{2} \quad\text{if}\quad z\in \left(\frac{3}{2},
\frac{m}{2}\right)\,,\quad\chi(z)\ge\frac{1}{2} \quad\text{if}\quad
z\in \left(-\frac{1}{2}, \frac{1}{2}\right)\,,
$$
and the set $E$ is defined by
$$
E=(\delta_1(\varepsilon), \delta_2(\varepsilon)):= \left\{\xi:
\frac{\sqrt{(\tilde{\mu}-\varepsilon)
(\omega^2+B^2)}}{\omega}<\xi<\frac{\sqrt{(\tilde{\mu}+\varepsilon)
(\omega^2+B^2)}}{\omega}\right\}\,,
$$
where $\tilde{\mu}:=\mu-\sqrt{\omega^2+B^2}$ and
$k,\,m,\,\alpha\,\in\mathbb{N}$ are positive integers to be chosen
later. Note that $\mathrm{supp}\,(\psi_{k,\alpha,m})\subset
[k,mk]\times[-k,k]$, and therefore
$$
\frac{\partial \psi_{k,\alpha,m}}{\partial x}(0+, y)=\frac{\partial
\psi_{k,\alpha,m}}{\partial x}(0-, y)=\lambda y
\psi_{k,\alpha,m}(0,y)=0\,, $$
which means that the functions $\psi_{k,\alpha,m}$ belong to the
domain of $H_\mathrm{Sm}(A)$ as needed.

First we observe that $\|\varphi_{k, \alpha,
m}\|_{L^2(\mathbb{R}^2)} \ge\frac{1}{8}$ because
\begin{eqnarray}
\hspace{-4em}\nonumber \int_{\mathbb{R}^2} |\varphi_{k, \alpha,
m}(x,y)|^2\,
\mathrm{d}x\,\mathrm{d}y \\
\hspace{-3em}\nonumber=\frac{1}{2\pi\,\mathrm{vol}(E)}\int_{\mathbb{R}^2}\left|\int_E
g\left(y-\frac{\xi B}{\omega^2+B^2}\right) e^{i \xi (x-\alpha k)}\,
\mathrm{d}\xi\right|^2\,\eta^2\left(\frac{x}{k}\right)\,
\chi^2\left(\frac{y}{k}\right)\,\mathrm{d}x\,\mathrm{d}y\\
\hspace{-3em}\nonumber=\frac{1}{2\pi\, \mathrm{vol}(E)}\int_k^{m
k}\,\int_{-k}^k\left|\int_E g\left(y-\frac{\xi
B}{\omega^2+B^2}\right) e^{i \xi (x-\alpha
k)}\,\mathrm{d}\xi\right|^2\,\eta^2\left(\frac{x}{k}\right)\,
\chi^2\left(\frac{y}{k}\right)\,\mathrm{d}x\,\mathrm{d}y\\
\hspace{-3em}\nonumber\ge\frac{1}{32\pi\,
\mathrm{vol}(E)}\int_{-k/2}^{k/2}\,\int_{3k/2}^{m k/2}\left|\int_E
g\left(y-\frac{\xi B}{\omega^2+B^2}\right) e^{i \xi (x-\alpha
k)}\,\mathrm{d}\xi\right|^2\,\mathrm{d}y\,\mathrm{d}x\\
\hspace{-3em}\label{int}=\frac{1}{32\pi\, \mathrm{vol}
(E)}\int_{-k/2}^{k/2}\,\int_{(3-2\alpha) k/2}^{(m-2\alpha)
k/2}\left|\int_E g\left(y-\frac{\xi B}{\omega^2+B^2}\right) e^{i \xi
x}\,\mathrm{d}\xi\right|^2\,\mathrm{d}y\,\mathrm{d}x\,.
\end{eqnarray}
By choosing  $\alpha=\alpha(k)$ and $m=m(\alpha, k)$ large enough
one is able to guarantee that for every  $y\in\left(-\frac{k}{2},
\frac{k}{2}\right)$ we have
\begin{eqnarray*}
\hspace{-1em}\frac{1}{2\pi}\int_{(3-2\alpha)k/2}^{(m-2\alpha)
k/2}\left|\int_E g\left(y-\frac{\xi B}{\omega^2+B^2}\right) e^{i \xi
x}\,\mathrm{d}\xi\right|^2\,\mathrm{d}x\\\ge\frac{1}{4\pi}\int_{\mathbb{R}}\left|\int_E
g\left(y-\frac{\xi B}{\omega^2+B^2}\right) e^{i \xi
x}\,\mathrm{d}\xi\right|^2\,\mathrm{d}x\\=\frac{1}{2}\int_Eg^2\left(y-\frac{x
B}{\omega^2+B^2}\right)\,\mathrm{d}x\,,
\end{eqnarray*}
where in the last step we have employed Plancherel formula. This
estimate together with (\ref{int}) gives for large $k$ the
inequalities
\begin{eqnarray*}
\hspace{-1em}\int_{\mathbb{R}^2} |\varphi_{k, \alpha, m}(x,y)|^2\,
\mathrm{d}x\,\mathrm{d}y\ge\frac{1}{32
\mathrm{vol}(E)}\int_{-k/2}^{k/2}\int_E \left|g\left(y-\frac{x
B}{\omega^2+B^2}\right)\right|^2\,\mathrm{d}x\,\mathrm{d}y\\\ge\frac{1}{32
\mathrm{vol}(E)}\int_E\,\int_{-k/2-x B/(\omega^2+B^2)}^{k/2-x
B/(\omega^2+B^2)}g^2(z)\,\mathrm{d}z\,\mathrm{d}x\\
\ge\frac{1}{64}\,\int_{\mathbb{R}}|g(z)|^2\,\mathrm{d}z=\frac{1}{64}.
\end{eqnarray*}

Our next aim is to show the validity of (\ref{Weyl ineq.}) with an
appropriate choice of $k,\,\alpha(k)$ and $m(\alpha, k)$. By a
straightforward calculation one gets
\begin{eqnarray}\nonumber
\hspace{-5em}\sqrt{2\pi}\, \frac{\partial^2\varphi_{k, \alpha,
m}}{\partial y^2} =\frac{1}{\sqrt{\mathrm{vol}(E)}} \left(\int_E
g''\left(y-\frac{\xi B}{\omega^2+B^2}\right) \,e^{i \xi (x-\alpha
k)}\,\mathrm{d}\xi\right)\,\eta\left(\frac{x}{k}\right)\,\chi\left(\frac{y}{k}\right)\\
\hspace{-4em}\label{derivative1}+\frac{2}{k \sqrt{\mathrm{vol}(E)}}
\left(\int_E g^\prime\left(y-\frac{\xi B}{\omega^2+B^2}\right)
\,e^{i \xi (x-\alpha
k)}\,\mathrm{d}\xi\right)\,\eta\left(\frac{x}{k}\right)\,\chi^\prime\left(\frac{y}{k}\right)\\
\hspace{-4em}\nonumber+\frac{1}{k^2 \sqrt{\mathrm{vol}(E)} }
\left(\int_E g\left(y-\frac{\xi B}{\omega^2+B^2}\right) \,e^{i \xi
(x-\alpha
k)}\,\mathrm{d}\xi\right)\,\eta\left(\frac{x}{k}\right)\,\chi''\left(\frac{y}{k}\right)\,,
\end{eqnarray}
and
\begin{eqnarray}\nonumber
\hspace{-5em}\sqrt{2\pi}\, \frac{\partial^2\varphi_{k, \alpha,
m}}{\partial x^2} =-\frac{1}{\sqrt{\mathrm{vol}(E)}}\left(\int_E
\xi^2 g\left(y-\frac{\xi B}{\omega^2+B^2}\right)\,e^{i \xi (x-\alpha
k)}\,\mathrm{d}\xi\right)\,\eta\left(\frac{x}{k}\right)\,\chi\left(\frac{y}{k}\right)\\
\hspace{-4em}\nonumber+\frac{2i}{k \sqrt{\mathrm{vol}(E)}}
\left(\int_E \xi g\left(y-\frac{\xi B}{\omega^2+B^2}\right)\,e^{i
\xi (x-\alpha
k)}\,\mathrm{d}\xi\right)\,\eta^\prime\left(\frac{x}{k}\right)\,\chi\left(\frac{y}{k}\right)\\
\hspace{-4em}\label{derivative2}+\frac{1}{k^2 \sqrt{\mathrm{vol}(E)}
} \left(\int_E g\left(y-\frac{\xi B}{\omega^2+B^2}\right)\,e^{i \xi
(x-\alpha
k)}\,\mathrm{d}\xi\right)\,\eta''\left(\frac{x}{k}\right)\,\chi\left(\frac{y}{k}\right)\,,\\
\hspace{-5em} \nonumber \sqrt{2\pi}\, y \frac{\partial\varphi_{k,
\alpha, m}}{\partial x} =\frac{i
y}{\sqrt{\mathrm{vol}(E)}}\left(\int_E \xi g\left(y-\frac{\xi
B}{\omega^2+B^2}\right) \,e^{i \xi (x-\alpha
k)}\,\mathrm{d}\xi\right)\,\eta\left(\frac{x}{k}\right)\,\chi\left(\frac{y}{k}\right)\\
\hspace{-4em}\label{calculations*}+\frac{y}{k
\sqrt{\mathrm{vol}(E)}}\left(\int_E  g\left(y-\frac{\xi
B}{\omega^2+B^2}\right) \,e^{i \xi (x-\alpha
k}\,\mathrm{d}\xi\right)\,\eta^\prime\left(\frac{x}{k}\right)\,\chi\left(\frac{y}{k}\right)
\,.\end{eqnarray}
We want to show that choosing $k$ sufficiently large one can make
the last two terms on the right-hand side of the first equation
(\ref{derivative1}) as small as one wishes in the $L^2$ norm, and
the same for the last two terms of the second equation
(\ref{derivative2}) and the last term of (\ref{calculations*}). This
follows from the following estimates,
\begin{eqnarray*}
\hspace{-4em}\frac{1}{k^2\, \mathrm{vol} (E)}
\int_{\mathbb{R}^2}\left|\int_E g^\prime\left(y-\frac{\xi
B}{\omega^2+B^2}\right) \,e^{i \xi (x-\alpha
k)}\,\mathrm{d}\xi\right|^2\,\eta^2\left(\frac{x}{k}\right)\,
\left(\chi^\prime\right)^2\left(\frac{y}{k}\right)\,\mathrm{d}x\,\mathrm{d}y
\\[.2em]
\hspace{-3em}\le\frac{\|\eta\|_\infty^2
\|\chi^\prime\|_\infty^2}{k^2\, \mathrm{vol} (E)}
\int_{\mathbb{R}^2}\left|\int_E g^\prime\left(y-\frac{\xi
B}{\omega^2+B^2}\right) \,e^{i \xi
x}\,\mathrm{d}\xi\right|^2\,\mathrm{d}x\,\mathrm{d}y\\[.2em]
\hspace{-3em}=\frac{\|\eta\|_\infty^2
\|\chi^\prime\|_\infty^2}{k^2\, \mathrm{vol}
(E)}\int_{\mathbb{R}}\int_E\left(g^\prime\left(y-\frac{x
B}{\omega^2+B^2}\right)\right)^2\,\mathrm{d}x\,\mathrm{d}y\\[.2em]
\hspace{-3em}=
\frac{\|\eta\|_\infty^2\,\|\chi^\prime\|^2_\infty}{k^2}
\int_{\mathbb{R}}\left(g^\prime\right)^2(z)\,\mathrm{d}z\,,
\end{eqnarray*}
where in the last step we employed again Plancherel formula, and
similarly,
\begin{eqnarray}
\hspace{-5em}\nonumber \frac{1}{k^4\, \mathrm{vol} (E)}
\int_{\mathbb{R}^2}\left|\int_E g\left(y-\frac{\xi
B}{\omega^2+B^2}\right) \,e^{i \xi (x-\alpha
k)}\,\mathrm{d}\xi\right|^2\,\eta^2\left(\frac{x}{k}\right)\,
\left(\chi''\right)^2\left(\frac{y}{k}\right)\,\mathrm{d}x\,\mathrm{d}y\\[.3em]
\hspace{-4em}\nonumber\le\frac{\|\eta\|_\infty^2\,\|\chi''\|^2_\infty
}{k^4}\int_{\mathbb{R}} g^2(z)\,\mathrm{d}z \\[.5em]
\hspace{-5em}\nonumber\frac{1}{k^2\, \mathrm{vol} (E)}
\int_{\mathbb{R}^2}\left|\int_E \xi g\left(y-\frac{\xi
B}{\omega^2+B^2}\right) \,e^{i \xi (x-\alpha
k)}\,\mathrm{d}\xi\right|^2\,\left(\eta^\prime\right)^2\left(\frac{x}{k}\right)\,
\chi^2\left(\frac{y}{k}\right)\,\mathrm{d}x\,\mathrm{d}y\\[.5em]
\hspace{-4em}\nonumber\le\frac{\|\eta^\prime\|_\infty^2\,\|\chi\|^2_\infty
\delta^2_2(\varepsilon)}{k^2}\int_{\mathbb{R}}
g^2(z)\,\mathrm{d}z\,,\\[.3em]
\hspace{-5em}\nonumber \frac{1}{k^4\, \mathrm{vol} (E)}
\int_{\mathbb{R}^2}\left|\int_E g\left(y-\frac{\xi
B}{\omega^2+B^2}\right) \,e^{i \xi (x-\alpha
k)}\,\mathrm{d}\xi\right|^2\,\left(\eta''\right)^2\left(\frac{x}{k}\right)\,
\chi^2\left(\frac{y}{k}\right)\,\mathrm{d}x\,\mathrm{d}y\\[.3em]
\hspace{-4em}\nonumber\le\frac{\|\eta''\|_\infty^2\,\|\chi\|^2_\infty
}{k^4}\int_{\mathbb{R}} g^2(z)\,\mathrm{d}z\,,\\[.5em]
\hspace{-5em}\nonumber \frac{1}{k^2\, \mathrm{vol} (E)}
\int_{\mathbb{R}^2}y^2\left|\int_E g\left(y-\frac{\xi
B}{\omega^2+B^2}\right) \,e^{i \xi (x-\alpha
k)}\,\mathrm{d}\xi\right|^2\,\left(\eta^\prime\right)^2\left(\frac{x}{k}\right)\,
\chi^2\left(\frac{y}{k}\right)\,\mathrm{d}x\,\mathrm{d}y\\[.3em]
\hspace{-4em}\label{integrals}
\le\frac{\|\eta^\prime\|_\infty^2\,\|\chi\|^2_\infty
}{k^2}\int_{\mathbb{R}}\left(|z|+\frac{\delta_2 (\varepsilon)\,
B}{\omega^2+B^2} \right)^2 g^2(z)\,\mathrm{d}z\,,
\end{eqnarray}
and consequently, all these integrals are at least
$\mathcal{O}(k^{-2})$ as $k\to\infty$. Then we can estimate the norm
on left-hand side of \eqref{Weyl ineq.} as
\begin{eqnarray*}
\hspace{-6em} \int_{\mathbb{R}^2}\left|H_{\mathrm{Sm}}(A)
\varphi_{k, \alpha, m}-\mu \varphi_{k, \alpha,
m}\right|^2(x,y)\,\mathrm{d}x\,\mathrm{d}y \\[.2em]
\hspace{-5em}
=\int_{\mathbb{R}^2}\biggl|-\frac{\partial^2\varphi_{k, \alpha,
m}}{\partial x^2}- \frac{\partial^2\varphi_{k, \alpha, m}}{\partial
y^2}+2 i By\frac{\partial\varphi_{k, \alpha, m}}{\partial x}
+(\omega^2+B^2)y^2\varphi_{k, \alpha, m}-\mu\varphi_{k, \alpha,
m}\biggr|^2\, \mathrm{d}x\,\mathrm{d}y \\[.2em] \hspace{-5em}
=\frac{1}{2\pi\, \mathrm{vol} (E)} \int_{\mathbb{R}^2}\biggl|\int_E
\biggl(-g''\left(y-\frac{\xi B}{\omega^2+B^2}\right) +\xi^2
g\left(y-\frac{\xi B}{\omega^2+B^2}\right)\,\\[.2em] \hspace{-4em}-2B y \xi
g\left(y-\frac{\xi B}{\omega^2+B^2}\right)+(\omega^2+B^2)y^2
g\left(y-\frac{\xi B}{\omega^2+B^2}\right)\\[.2em] \hspace{-4em} -\mu g\left(y-\frac{\xi
B}{\omega^2+B^2}\right)\biggr)\,e^{i \xi (x-\alpha
k)}\,\mathrm{d}\xi\biggr|^2\,\eta^2\left(\frac{x}{k}\right)\,
\chi^2\left(\frac{y}{k}\right)\,\mathrm{d}x\,\mathrm{d}y+\mathcal{O}(k^{-2})
\end{eqnarray*}
\begin{eqnarray*}
\hspace{-5em} =\frac{1}{2\pi \mathrm{vol} (E)}
\int_{\mathbb{R}^2}\biggl|\int_E \biggl(-g''\left(y-\frac{\xi
B}{\omega^2+B^2}\right)+\biggl((B^2+\omega^2) \left(y-\frac{\xi
B}{\omega^2+B^2}\right)^2\\[.2em]
\hspace{-4em} +\frac{\omega^2
\xi^2}{\omega^2+B^2}-\mu\biggr)g\left(y-\frac{\xi
B}{\omega^2+B^2}\right) \biggr)\,e^{i \xi (x-\alpha
k)}\,\mathrm{d}\xi\biggr|^2\,\eta^2\left(\frac{x}{k}\right)\,
\chi^2\left(\frac{y}{k}\right)\,\mathrm{d}x\,\mathrm{d}y+\mathcal{O}(k^{-2})\\[.2em]
\hspace{-5em} =\frac{1}{2\pi \mathrm{vol} (E)}
\int_{\mathbb{R}^2}\biggl|\int_E \left(\frac{\omega^2
\xi^2}{\omega^2+B^2}-\tilde{\mu}\right) g\left(y-\frac{\xi
B}{\omega^2+B^2}\right) \biggr)\,e^{i \xi (x-\alpha
k)}\,\mathrm{d}\xi\biggr|^2\,\\[.2em]
\hspace{-4em} \times\eta^2\left(\frac{x}{k}\right)\,
\chi^2\left(\frac{y}{k}\right)\,\mathrm{d}x\,\mathrm{d}y+\mathcal{O}(k^{-2})\\[.2em]
\hspace{-5em}\le \frac{\|\eta\|_\infty^2 \|\chi\|_\infty^2}{2\pi
\mathrm{vol} (E)}\int_{\mathbb{R}^2}\biggl|\int_E
\left(\frac{\omega^2 \xi^2}{\omega^2+B^2}-\tilde{\mu}\right)
g\left(y-\frac{\xi B}{\omega^2+B^2}\right) \biggr)\,e^{i \xi
x}\,\mathrm{d}\xi\biggr|^2\,\mathrm{d}x\,\mathrm{d}y
+\mathcal{O}(k^{-2})  \\[.2em]
\hspace{-5em} \le\frac{\|\eta\|_\infty^2
\|\chi\|_\infty^2}{\mathrm{vol} (E)} \int_{E\times\mathbb{R}}
\left(\frac{\omega^2 x^2}{\omega^2+B^2}-\tilde{\mu}\right)^2
g^2\left(y-\frac{x B}{\omega^2+B^2}\right)
\,\mathrm{d}x\,\mathrm{d}y +\mathcal{O}(k^{-2})\\[.2em]
\hspace{-5em} \le\frac{\varepsilon^2 \|\eta\|_\infty^2
\|\chi\|_\infty^2}{\mathrm{vol} (E)} \int_{E\times\mathbb{R}}
g^2\left(y-\frac{xB}{\omega^2+B^2}\right)\,\mathrm{d}x\,\mathrm{d}y
+\mathcal{O}(k^{-2})\\[.2em]
\hspace{-5em}\le\varepsilon^2 \|\eta\|_\infty^2 \|\chi\|_\infty^2
\int_{\mathbb{R}} g^2(z)\,\mathrm{d}z+\mathcal{O}(k^{-2})\,.
\end{eqnarray*}
Hence choosing a large enough $k$ we can achieve that
\begin{align}
\nonumber \hspace{-.6em}\int_{\mathbb{R}^2}|H_{\mathrm{Sm}}
(A)\varphi_{k, \alpha, m}-\mu\varphi_{k, \alpha, m}|^2(x,y)\:
\mathrm{d}x\,\mathrm{d}y & <64\varepsilon^2 \|\eta\|_\infty^2
\|\chi\|_\infty^2 \int_{\mathbb{R}^2}|\varphi_{k, \alpha,
m}|^2\,\mathrm{d}x\,\mathrm{d}y+\mathcal{O}(k^{-2})\\[.2em]
\label{final1} & \le65\varepsilon^2 \|\eta\|_\infty^2
\|\chi\|_\infty^2 \int_{\mathbb{R}^2}|\varphi_{k, \alpha,
m}|^2\,\mathrm{d}x\,\mathrm{d}y\,.
\end{align}
To complete the proof we choose a sequence
$\{\varepsilon_j\}_{j=1}^\infty$ such that $\varepsilon_j\searrow0$
holds as $j\to\infty$ and to any given $j$ we construct a function
$\left\{\varphi_{\varepsilon_j}\right\}_{j=1}^\infty=\left\{\varphi_{k(\varepsilon_j),
\alpha(k(\varepsilon_j)), m(\alpha(k(\varepsilon_j)),
k(\varepsilon_j))}\right\}_{j=1}^\infty$ with the parameters chosen
in such a way that $k(\varepsilon_j)
>m(\alpha(k(\varepsilon_{j-1})),
k(\varepsilon_{j-1}))\,k(\varepsilon_{j-1})$. The norms of
$H_{\mathrm{Sm}} (A)\varphi_{\varepsilon_j}$ satisfy the inequality
(\ref{final1}) with $65\varepsilon_j^2 \|\eta\|_\infty^2
\|\chi\|_\infty^2 \|\varphi_{\varepsilon_j}\|_{L^2(\mathbb{R}^2)}^2$
on the right-hand side, and at the same time the sequence converges
by construction weakly to zero.
\end{proof}

\subsection{Subcritical case: the essential spectrum threshold}
\label{ss:Smil-subcr-thresh}

Next we are going to show that in the subcritical case the inclusion
in Theorem~\ref{th:subcritess} is in fact an equality.
\begin{theorem} Let $\lambda>-2
\omega$, then the spectrum of $H_{\mathrm{Sm}} (A)$ below
$\sqrt{\omega^2+B^2}$ is purely discrete.
\end{theorem}
\begin{proof}
We employ a Neumann bracketing in a way similar to \cite{BE17}. Let
$h^{(\pm)}_n (A)$ and $h_0 (A)$ be the Neumann restrictions of
operator $H_{\mathrm{Sm}}(A)$ to the regions
$$
G^{(\pm)}_n=\mathbb{R} \times \{y:\: \pm y \ge n\}$$
and $G_0=\mathbb{R}\times [-n, n]$, where $n\in\mathbb{N}$ will be
chosen later. In view of the minimax principle \cite[Secs.~XIII.1
and XIII.15]{RS78} we have the inequality
\begin{equation}\label{N}
H_{\mathrm{Sm}}(A)\ge \left(h^{(+)}_n (A)\oplus h^{(-)}_n
(A)\right)\oplus h^{(0)} (A)\,.
\end{equation}
To prove the claim we are going first to demonstrate that for a
sufficiently large $n$ the spectra of $h_n^{(\pm)}(A)$ below
$\sqrt{\omega^2+B^2}$ are empty, and secondly, that for any
$\Lambda<\sqrt{\omega^2+B^2}$ the spectrum of $h_0(A)$ below
$\Lambda$ is purely discrete.

The quadratic form $Q(h_n^{(\pm)} (A))$ corresponding to
$h_n^{(\pm)} (A)$ coincides with
\begin{align*}
& Q(h_n^{(\pm)} (A))=\int_{\mathbb{R}}\int_{\pm y\ge n}
\left|\frac{\partial u}{\partial x}-i B y
u\right|^2\,\mathrm{d}x\,\mathrm{d}y+\int_{\mathbb{R}}\int_{\pm y\ge
n} \left|\frac{\partial u}{\partial
y}\right|^2\,\mathrm{d}x\,\mathrm{d}y \\[.3em]
& \qquad +\omega^2\int_{\mathbb{R}}\int_{\pm y\ge n}y^2
|u|^2\,\mathrm{d}x\,\mathrm{d}y +\lambda\int_{\mathbb{R}}\int_{\pm
y\ge n}y |u(0, y)|^2\,\mathrm{d}y \\[.3em] & \quad
\ge\int_{\mathbb{R}}\int_{\pm y\ge
n} \left|\frac{\partial u}{\partial x}-i B y
u\right|^2\,\mathrm{d}x\,\mathrm{d}y
+\omega^2\int_{\mathbb{R}}\int_{\pm y\ge n}y^2
|u|^2\,\mathrm{d}x\,\mathrm{d}y+\lambda \int_{\mathbb{R}}\int_{\pm
y\ge n}y |u(0, y)|^2\,\mathrm{d}y
\end{align*}
defined on $\mathcal{H}^1(\mathbb{R}\times \{y: \pm y\ge n\})$. As
before the quadratic forms
$$
u(\cdot,y)\mapsto \int_{\mathbb{R}}\left|i \frac{\partial
u}{\partial x}-B y u\right|^2+\omega^2 y^2
|u|^2\,\mathrm{d}x+\lambda y |u(0, y)|^2
$$
are for any fixed $y$ unitarily equivalent to $y^2 L$ if $y>0$ and
$y^2\widetilde{L}$ if $y<0$, where $\widetilde{L}$ is given in
(\ref{widetildeL}). Thus we have
\begin{eqnarray*}
\inf \sigma (h_n^{\pm})\ge n^2 \inf \sigma (L)=n^2
\left(\omega^2-\frac{\lambda^2}{4}\right)\quad\text{if}\quad y>0\,,
\\[.3em]
\inf \sigma (h_n^{\pm})\ge n^2 \inf \sigma (\widetilde{L})=n^2\,;
\omega^2 \quad\text{if}\quad y<0
\end{eqnarray*}
recall that we suppose $\lambda<0$. This concludes the proof of our
first claim provided one chooses a sufficiently large $n$. It
remains to inspect the essential spectrum of $h_0 (A)$. To this aim
we employ the following auxiliary result.
\begin{lemma} \label{pertlemma}
Under our assumptions
\begin{equation}\label{stability}
\inf \sigma_{\mathrm{ess}} (h_0 (A)) =\inf
\sigma_{\mathrm{ess}}(\tilde{h}_0 (A))\,,
\end{equation}
where $\tilde{h}_0 (A)$ is the Neumann operator $\left(i
\frac{\partial}{\partial x}- B y\right)^2-\frac{\partial^2}{\partial
y^2}+\omega^2 y^2$ defined  on $\mathcal{H}^1(G_0)$.
\end{lemma}
\begin{proof}
We are going to show that for any real number $\mu$ from the
resolvent sets of both $h_0 (A)$  and $\tilde{h}_0 (A)$  the
operator
\begin{equation}\label{W}
W:=\left(h_0 (A)-\mu \mathbb{I}\right)^{-1}-\left(\tilde{h}_0
(A)-\mu \mathbb{I}\right)^{-1}
\end{equation}
is compact in $L^2(G_0)$. We proceed as in \cite{BEL13}; for any
fixed $f, g\in L^2(G_0)$ we put $u:=\left(h_0 (A)-\mu
\mathbb{I}\right)^{-1} f$ and $v:=\big(\tilde{h}_0 (A)-\mu
\mathbb{I}\big)^{-1} g$. Then
\begin{eqnarray*}
(Wf, g)=\left(\left(h_0 (A)-\mu \mathbb{I}\right)^{-1} f,
g\right)-\big(\big(\tilde{h}_0 (A)-\mu \mathbb{I}\big)^{-1} f,
g\big)\\ \quad =(u, g)-(f,
v)=\big(u,\big({h}_0(A)-\mu\mathbb{I}\big) v\big) -\left(\left(h_0
(A)-\mu \mathbb{I}\right) u,v\right)\\\quad =\big(\tilde{h}_0 (A) u,
v\big)-\left(h_0 (A) u,v\right) =-\alpha \int_{-n}^n y u(0, y)
\overline{v(0,y)}\,\mathrm{d}y\,.
\end{eqnarray*}
Let $\{f_n\}_{n=1}^\infty$ and $\{g_n\}_{n=1}^\infty$ be  bounded
sequences in $L^2(G_0)$. Then it follows from the Sobolev trace
theorem \cite{M00} that the $\mathcal{H}^1(G_0)$ bounded functions
$u_n:=\left(h_0 (A)-\mu \mathbb{I}\right)^{-1} f_n$ and
$v_n=\left(\tilde{h}_0 (A)-\mu \mathbb{I}\right)^{-1} g_n$ are also
bounded in $\mathcal{H}^{1/2}(x=0,\,|y|\le n).$ In view of the
compact embedding $\mathcal{H}^{1/2}(x=0,\,|y|\le n) \hookrightarrow
L^2(x=0,\,|y|\le n)$ the sequences $\{u_n\}$ and
$\{v_n\}_{n=1}^\infty$ are compact in $L^2(x=0,\,|y|\le n)$. In
combination with the inequality
\begin{eqnarray*}
\left(W(f_n-f_m), W(f_n-f_m)\right)\\
\quad \le|\alpha| \,n\,\sqrt{\int_{-n}^n  |(u_n-u_m)(0,
y)|^2\,\mathrm{d}y}\,\,\sqrt{\int_{- n}^n
|(\tilde{v}_n-\tilde{v}_m)(0, y)|^2\,\mathrm{d}y}
\end{eqnarray*}
with $\tilde{v}_n=\big(\tilde{h}_0 (A)-\mu \mathbb{I}\big)^{-1}
(W(f_n-f_m))$, this establishes our claim.
\end{proof}

Let us now return to the proof of the theorem. In view of the lemma
it is sufficient to inspect the threshold of
$\sigma_\mathrm{ess}\big(\tilde{h}_0 (A)\big)$. Using the partial Fourier
transformation $\mathrm{F}_\xi (u(x, y))$ given by
\begin{equation}\label{Fourier}
\mathrm{F}_\xi (u(x,
y))=\widehat{u}(x, y)=\frac{1}{\sqrt{2 \pi}} \int_{\mathbb{R}}
u(\xi, y) e^{-i \xi x}\,\mathrm{d}\xi
\end{equation}
and the Landau gauge for the vector potential, $A=(-B y, 0)$, one is
able to rewrite the quadratic form $Q_0$ of $\tilde{h}_0 (A)$ as
\begin{eqnarray*}
\hspace{-6em}Q_0 (u)=\int_{\mathbb{R}\times [-n, n]} \left(\left|i
\nabla u+A u)\right|^2(x, y)+\omega^2 y^2|u|^2(x,
y)\right)\,\mathrm{d}x\,\mathrm{d}y\\ \hspace{-5.5em}
=\int_{\mathbb{R}\times [-n, n]} \left[\left|\frac{\partial
\widehat{u}}{\partial y}\right|^2(\xi, y)+\biggl((\omega^2+B^2)
\left(y-\frac{\xi B}{\omega^2+B^2}\right)^2 +\frac{\omega^2
\xi^2}{\omega^2+B^2}\biggr)|\widehat{u}|^2(\xi,y)
\right]\,\mathrm{d}\xi\,\mathrm{d}y\,.
\end{eqnarray*}
Thus $\tilde{h}_0 (A)$ is unitarily equivalent to the direct
integral
$$
\int_{\mathbb{R}}^{\oplus} h(\xi)\,\mathrm{d}\xi\,,
$$
where the fibers $h(\xi) =-\frac{\mathrm{d}^2}{\mathrm{d}y^2}
+(\omega^2+B^2) \left(y-\frac{\xi B}{\omega^2+B^2}
\right)^2+\frac{\omega^2 \xi^2}{\omega^2+B^2}$ are one-dimensional
Neumann operators defined on $L^2(-n, n)$. By a simple change of
variables we arrive at the Neumann harmonic oscillators $l
(\xi)=-\frac{\mathrm{d}^2}{\mathrm{d}y^2}+(\omega^2+B^2)
y^2+\frac{\omega^2 \xi^2}{\omega^2+B^2}$  on the interval
$\left[-n-\frac{\xi B}{\omega^2+B^2}, n-\frac{\xi
B}{\omega^2+B^2}\right]$. Similarly as in \cite{BE17} one can check
that
$$
\inf \sigma (l (\xi))\ge \sqrt{\omega^2+B^2}+\mathcal{O}(n^{-1})
$$
holds for large $n$ uniformly in $\xi\in\mathbb{R}$. The spectrum of $\sigma (\tilde{h}_0 (A))$ is determined by those of the fiber operators \cite[Sec.~XIII.16]{RS78}, in particular we get
\begin{equation}\label{sigmaess}
\sigma (\tilde{h}_0 (A))=\sigma_{\mathrm{ess}}  (\tilde{h}_0 (A))\subset \left[\sqrt{\omega^2+B^2}+\mathcal{O}(n^{-1}), \infty\right)\,.
\end{equation}
By virtue of Lemma~\ref{pertlemma} this yields
$$
\inf \sigma_{\mathrm{ess}} (h_0 (A))\ge \sqrt{\omega^2+B^2}+\mathcal{O}(n^{-1})\,,
$$
and therefore for any $\Lambda<\sqrt{\omega^2+B^2}$ one can choose $n$ large enough to ensure that the spectrum of $h_0(A)$ below $\Lambda$ is purely discrete which concludes the proof.
\end{proof}

\subsection{Subcritical case: existence of the discrete spectrum}
\label{ss:Smil-subcr-exist}

The above results localize exactly the essential spectrum, however, they tell us nothing about the existence of the discrete spectrum. This is the question we are going to address now.
\begin{theorem}
Let $\lambda>-2\omega$, then the discrete spectrum of $H_{\mathrm{Sm}} (A)$ is nonempty and contained in the interval $(0, \sqrt{\omega^2+B^2})$.
\end{theorem}
\begin{proof}
To demonstrate the non-emptiness of the discrete spectrum one needs to construct a normalized function $u\in \mathrm{Dom}(Q(H_{\mathrm{Sm}} (A)))$ such that $Q(H_{\mathrm{Sm}} (A))(u)<\sqrt{\omega^2+B^2}.$ On the other hand, since $\lambda>-2\omega>-2\sqrt{\omega^2+B^2}$ the non-magnetic operator
$\widetilde{H}=-\Delta+(\omega^2+B^2) y^2+\lambda y\delta(x)$ has a nonempty finite set of eigenvalues  below $\sqrt{\omega^2+B^2}$, and moreover, the corresponding eigenfunctions can be chosen real-valued \cite{S04}. It is easy to see that for any such eigenfunction $u$ we have
$$
Q(H(A))(u)=Q(\widetilde{H})(u)<\sqrt{\omega^2+B^2}\,,
$$
and since by Proposition~\ref{subcr-posit} operator $H_{\mathrm{Sm}} (A)$ is positive, the claimed is proved.
\end{proof}

\subsection{The supercritical case} \label{ss:Smil-supercrit}

Let us now turn to the situation when the coupling constant surpasses the critical value. As discussed in the opening of Sec.~\ref{s:Smil}, we know that operator $H_{\mathrm{Sm}}(A)\upharpoonright \mathcal{D}_0$ has self-adjoint extensions and in the following we will use the symbol $H_{\mathrm{Sm} (A)}$ for any of them.
\begin{theorem} \label{th:supercrit}
$\sigma(H_{\mathrm{Sm}} (A))=\mathbb{R}$ holds provided $\lambda<-2\omega$.
\end{theorem}
\begin{proof}
To check that any real number $\mu$ belongs to the spectrum of $H_{\mathrm{Sm}} (A)$ we employ Weyl's criterion finding a sequence $\{\psi_k\}_{k=1}^\infty\subset D(H_{\mathrm{Sm}} (A))$ such that $\|\psi_k\|=1$ satisfying
$$
\|H_{\mathrm{Sm}} (A)\psi_k-\mu\psi_k\|\to0 \quad\text{as}\quad k\to\infty\,;
$$
note that one need not require that $\{\psi_k\}$ contains no convergent subsequence because the spectrum covering the whole real axis cannot be anything else than essential. To this aim we modify the method of \cite{BE14} without repetition of the parts that do not change. With scaling transformations in mind we may suppose that $\inf\sigma(L)=\omega^2-\frac14\lambda^2=-1$ corresponding to the single eigenvalue of the operator which is simple and associated with the normalized eigenfunction $h=\sqrt{\frac{|\lambda|}{2}} e^{-|\lambda|\,|t|/2}$.

As in \cite{BE14} we will first show that $0\in\sigma(H_{\mathrm{Sm}} (A))$. We fix an $\varepsilon>0$ and choose a positive integer $k=k(\varepsilon)$ to which we associate a function $\chi_k\subset C_0^2(1,k)$ such that
\begin{equation}\label{chi.conditions}
\int_1^k\frac{1}{z}\chi_k^2(z)\,\mathrm{d}z=1 \quad\text{and}\quad
\int_1^k z(\chi'_k(z))^2\,\mathrm{d}z<\varepsilon\,;
\end{equation}
we know from \cite{BE14} that such functions can be constructed. Then we define
\begin{equation}
\label{sequence} \psi_k(x,y):=h(x
y)\,\e^{iy^2/2}\chi_k\left(\frac{y}{n_k}\right) +\frac{f(x
y)}{y^2}\, \e^{iy^2/2}\chi_k\left(\frac{y}{n_k}\right)\,,
\end{equation}
where the smooth function $f\in\mathrm{Dom} (L)$ and the positive
integer $n_k\in\mathbb{N}$ will be chosen later. The functions
\eqref{sequence} belong to $\mathrm{Dom}(H_{\mathrm{Sm}}(A))$ by
construction and have the following property \cite{BE14}.
\begin{lemma}
$\|\psi_k\|_{L^2(\mathbb{R}^2)}\ge\frac{1}{2}$ holds provided $n_k$ is large enough.
\end{lemma}
\noindent Next we have to show that the functions $\psi_k$
approximate the generalized eigen\-function corresponding to zero
energy.
\begin{lemma}
$\|H_{\mathrm{Sm}} (A)\psi_k\|_{L^2(\mathbb{R}^2)}^2<c\varepsilon$ holds with a $c$ independent of $k$ provided $n_k$ is large enough.
\end{lemma}
\begin{proof}
We have to estimate the following integral,
$$
\int_{\mathbb{R}^2}\left|H_{\mathrm{Sm}}
(A)\psi_k\right|^2(x,y)\,\mathrm{d}x\,\mathrm{d}y
=\int_{\mathbb{R}^2}\biggl|-\frac{\partial^2\psi_k}{\partial x^2}-
\frac{\partial^2\psi_k}{\partial y^2}+2i B x \frac{\partial
\psi_k}{\partial y}+B^2 x^2 \psi_k +\omega^2 y^2\psi_k\biggr|^2\,
\mathrm{d}x\,\mathrm{d}y\,.
$$
We know from \cite{BE14} that the claim is valid if $B=0$, hence it
remains to deal with the additional terms associated with the
magnetic field. We have
\begin{align*}
\frac{\partial \psi_k}{\partial y}=&\biggl(x h^\prime(x
y)\,\chi_k\left(\frac{y}{n_k}\right)+\frac{1}{n_k}h(x
y)\,\chi_k^\prime\left(\frac{y}{n_k}\right)+\frac{x}{y^2} f^\prime(x
y)\,\chi_k\left(\frac{y}{n_k}\right) +\frac{i}{y}f(x
y)\,\chi_k\left(\frac{y}{n_k}\right)\\ & +\frac{1}{n_k y^2}f(x
y)\,\chi_k^\prime\left(\frac{y}{n_k}\right)-\frac{2}{y^3}f(x
y)\,\chi_k\left(\frac{y}{n_k}\right)+i y h(x
y)\,\chi_k\left(\frac{y}{n_k}\right)\biggr)\,\e^{iy^2/2}
\end{align*}
which allows us to check that choosing $n_k$ large enough one can
make norms of all the terms in the expression of $x \frac{\partial
\psi_k}{\partial y}$ except of the last one small enough, because
\begin{align*}
\int_{\mathbb{R}^2} \left|x^2 h^\prime(x y)\,\e^{iy^2/2}\chi_k
\left(\frac{y}{n_k}\right)\right|^2\,\mathrm{d}x\,\mathrm{d}y
&\le\frac{1}{n_k^4}\int_{\mathbb{R}}|h^\prime(t)|^2\,\mathrm{d}t\int_1^k|\chi_k(z)|^2\,\mathrm{d}z\,,\\[.2em]
\frac{1}{n_k^2}\int_{\mathbb{R}^2}\left|h(x
y)\,\e^{iy^2/2}\chi^\prime_k \left(\frac{y}{n_k}\right)
\right|^2\,\mathrm{d}x\,\mathrm{d}y&\le
\frac{1}{n_k^2}\int_{\mathbb{R}}|h(t)|^2\,\mathrm{d}t\,\int_1^k|\chi^\prime_k(z)|^2\,\mathrm{d}z\,,
\\[.2em]\int_{\mathbb{R}^2}\left|\frac{x}{y^2}f^\prime(x
y)\,\e^{iy^2/2}\chi_k\left(\frac{y}{n_k}\right)\right|^2\,\mathrm{d}x\,\mathrm{d}y
&\le\frac{1}{n_k^4}\int_{\mathbb{R}}|f^\prime(t)|^2\,\mathrm{d}t\int_1^k|\chi_k(z)|^2\,\mathrm{d}z\,,\\[.2em]
\int_{\mathbb{R}^2}\left|\frac{i}{y} f(x
y)\,\e^{iy^2/2}\chi_k\left(\frac{y}{n_k}\right)\right|^2\,\mathrm{d}x\,\mathrm{d}y
&\le\frac{1}{n_k^2}\int_{\mathbb{R}}|f(t)|^2\,\mathrm{d}t\int_1^k|\chi_k(z)|^2\,\mathrm{d}z\,,\\[.2em]
\int_{\mathbb{R}^2}\left|\frac{1}{n_k y^2} f(x
y)\,\e^{iy^2/2}\chi_k^\prime\left(\frac{y}{n_k}\right)\right|^2\,\mathrm{d}x\,\mathrm{d}y
&\le\frac{1}{n_k^6}\int_{\mathbb{R}}|f(t)|^2\,\mathrm{d}t\,\int_1^k|\chi^\prime_k(z)|^2\,\mathrm{d}z\,,\\[.2em]
\int_{\mathbb{R}^2}\left|\frac{2}{y^3} f(x
y)\,\e^{iy^2/2}\chi_k\left(\frac{y}{n_k}\right)\right|^2\,\mathrm{d}x\,\mathrm{d}y&\le
\frac{4}{n_k^6}\int_{\mathbb{R}}|f(t)|^2\,\mathrm{d}t\,\int_1^k|\chi_k(z)|^2\,\mathrm{d}z\,.
\end{align*}
In a similar way one can check that for large $n_k$ the integral
$\int_{\mathbb{R}^2}\left|x^2 \psi_k\right|^2\,
\mathrm{d}x\,\mathrm{d}y$ is less than $\varepsilon$. This yields
the estimate
\begin{align*} \int_{\mathbb{R}^2} & \left|H_{\mathrm{Sm}}
(A)\psi_k\right|^2(x,y)\,\mathrm{d}x\,\mathrm{d}y\le
21\int_{n_k}^{kn_k}\int_{\mathbb{R}}\biggl|y^2\left(h''(xy)-\omega^2h(xy)-
h(xy)\right)\chi_k\left(\frac{y}{n_k}\right) \\[.2em] &
+ih(xy)\chi_k\left(\frac{y}{n_k}\right)
+f''(xy)\chi_k\left(\frac{y}{n_k}\right)+2ixyh'(xy)\chi_k\left(\frac{y}{n_k}\right)
+\frac{2iy}{n_k}h(xy)\chi'_k\left(\frac{y}{n_k}\right)
\\[.2em] & -f(xy)\chi_k\left(\frac{y}{n_k} \right)
-\omega^2\,f(xy)\chi_k\left(\frac{y}{n_k}\right) +2B x y h(x
y)\chi_k\left(\frac{y}{n_k}\right)\biggr|^2\,\mathrm{d}x\,\mathrm{d}y
+21\varepsilon\,,
\end{align*}
where the coefficient in front of the integrals comes from the
number of the summands. Using the fact that $Lh=-h$ and applying the
Cauchy inequality, the above inequality implies
\begin{eqnarray*}
\hspace{-5em}\lefteqn{\frac{1}{21}\int_{\mathbb{R}^2}|H_{\mathrm{Sm}}
(A)\psi_k|^2(x,y)\,\mathrm{d}x\,\mathrm{d}y<\int_{n_k}^{kn_k}
\int_{\mathbb{R}}\biggl|\biggl(f''(xy)+2ixyh'(xy)+ih(xy)-f(xy)}
\\&& -\omega^2\,f(xy)+2B x y h(x
y)\biggr)\chi_k\left(\frac{y}{n_k}\right)
   +\frac{2iy}{n_k}h(xy)\chi_k'
\left(\frac{y}{n_k}\right)\biggr|^2\,\mathrm{d}x\,\mathrm{d}y
+\varepsilon \\ && \hspace{-1em}
\le2\int_1^k\frac{1}{z}|\chi_k(z)|^2\,\mathrm{d}z\,\int_{\mathbb{R}}
\biggl|-f''(t)+f(t)\left(1+\omega^2)\right)-2ith'(t)\\&&-ih(t)-2B t
h(t)\biggr|^2 \mathrm{d}t
+8\int_1^kz|\chi_k'(z)|^2\,\mathrm{d}z+\varepsilon \\ &&
\hspace{-1em} \le2\int_{\mathbb{R}}
\biggl|-f''(t)+f(t)\left(1+\omega^2\right)-2ith'(t)-ih(t)-2B t
h(t)\biggr|^2 \,\mathrm{d}t+9\varepsilon\,.
\end{eqnarray*}
It is easy to check that
$$
\int_{\mathbb{R}}(2ith'(t)+ih(t)+2B t h(t)) h(t)\,\mathrm{d}t=0\,,
$$
which together with the simplicity of the eigenvalue $-1$
establishes the existence of the solution for the differential
equation
\begin{equation}\label{diff.eq.}
-f''(t)+f(t)\left(1+\omega^2\right)-2ith'(t)-ih(t)-2B
t h(t)=0
\end{equation}
belonging to $\mathrm{Dom}(L)$ which will use in the Ansatz
\eqref{sequence}. With this choice the last integral in the above
estimate vanishes, which gives
\begin{equation}
\label{final2} \int_{\mathbb{R}^2}|H_{\mathrm{Sm}}
(A)\psi_k|^2(x,y)\,\mathrm{d}x\,\mathrm{d}y\le 189\varepsilon
\end{equation}
concluding this the proof of the lemma.
\end{proof}

Now we can complete the proof of the theorem. We fix a sequence
$\{\varepsilon_j\}_{j=1}^\infty$ such that $\varepsilon_j\searrow0$
holds as $j\to\infty$ and to any $j$ we construct a function
$\psi_{k(\varepsilon_j)}$. As we have mentioned it is not necessary
that such a sequence converges weakly to zero, but we can achieve
that at no extra expense by choosing the corresponding numbers in
such a way that $n_{k(\varepsilon_j)}>k(\varepsilon_{j-1})
n_{k(\varepsilon_{j-1})}$. The norms of $H_{\mathrm{Sm}}
(A)\psi_{k(\varepsilon_j)}$ satisfy inequality which (\ref{final2})
with $189\varepsilon_j$ on the right-hand side, which yields the
desired result.

For any nonzero real number $\mu$ we use the same procedure
replacing \eqref{sequence} with
$$
\psi_k(x,y)=h(xy)\,\e^{i\epsilon_\mu(y)}
\chi_k\left(\frac{y}{n_k}\right)+
\frac{f(xy)}{y^2}\,\e^{i\epsilon_\mu(y)}
\chi_k\left(\frac{y}{n_k}\right)\,,
$$
where $\epsilon_\mu(y):= \displaystyle{\int_{\sqrt{|\mu|}}^y
\sqrt{t^2+\mu}\,\mathrm{d}t}$ and the functions $f,\,\chi_k$ are the
same as above. Repeating the estimates with the modified exponential
function one can check that to any positive $\varepsilon$ one can
choose the integer $n_k$ large enough so that
\begin{align*}
\biggl\|\frac{\partial^2\psi_k}{\partial y^2} &
\,\e^{-i\varepsilon_\mu(y)}-2i B x \frac{\partial \psi_k}{\partial
y} \,\e^{-i\varepsilon_\mu(y)} +\mu\psi_k
\,\e^{-i\varepsilon_\mu(y)} \\[.2em] & -
\e^{-iy^2/2}\biggl(\frac{\partial^2}{\partial y^2}\biggl(\psi_k
\e^{-i\varepsilon_\mu(y)+iy^2/2}\biggr) -2i B x
\frac{\partial}{\partial y}\biggl(\psi_k
\e^{-i\varepsilon_\mu(y)+iy^2/2}\biggr)\biggr)\biggr\|_{L^2(\mathbb{R}^2)}
<\varepsilon
\end{align*}
holds. Using further the identity $\frac{\partial^2\psi_k}{\partial
x^2}\,\e^{-i\varepsilon_\mu(y)}= \e^{-iy^2/2}
\frac{\partial^2}{\partial x^2}\bigl(\psi_k\,
\e^{-i\varepsilon_\mu(y)+iy^2/2}\bigr)$ we get
\begin{eqnarray*}
\hspace{-4em}\lefteqn{\|H_{\mathrm{Sm}}
(A)\psi_k-\mu\psi_k\|_{L^2(\mathbb{R}^2)}=\Bigl\|(H_{\mathrm{Sm}}
(A)\psi_k) \e^{-i\varepsilon_\mu(y)} -\mu\psi_k
\,\e^{-i\varepsilon_\mu(y)}\Bigr\|_{L^2(\mathbb{R}^2)}} \\ &&
\hspace{-1em} <\Bigl\|\e^{-iy^2/2}H_{\mathrm{Sm}} (A)\Bigl (\psi_k
\,\e^{-i\varepsilon_\mu(y)+iy^2/2}\Bigr)\Bigr\|_{L^2(\mathbb{R}^2)}+\varepsilon\,;
\end{eqnarray*}
now we can use the result of the first part of proof to conclude the
proof.
\end{proof}

\begin{remark}
{\rm In view of Theorem~\ref{th:subcritess} we could have restricted
our attention to the numbers $\mu<\sqrt{\omega^2+B^2}$ only.
Avoiding this restriction makes sense, however, showing that in the
supercritical case one can construct \emph{for any $\mu$} a Weyl
sequence with the support in the vicinity of the $y$ axis. Looking
at the problem from the dynamical point of view as in \cite{Gu11},
this fact is connected with the existence of states escaping to
infinity along the singular channel.}
\end{remark}

\subsection{The critical case}
\label{ss:Smil-crit}

If the two competing forces are in exact balance,
$\lambda=-2\omega$, the quadratic form is still positive by
Proposition~\ref{subcr-posit}, hence $H_{\mathrm{Sm}}(A)$ can be
defined as the Friedrich's extension of the operator initially
defined on set $\mathcal{D}_0$.
\begin{theorem}\label{th:critess}
Let $\lambda=-2\omega$, then under the stated assumptions we have
$$\sigma (H_{\mathrm{Sm}} (A))=\sigma_{\mathrm{ess}}
(H_{\mathrm{Sm}})=[0, \infty)\,.
$$
\end{theorem}
\begin{proof}
In view of Proposition~\ref{subcr-posit} it is enough to show that
$\sigma_{\mathrm{ess}} (H_{\mathrm{Sm}}) (A)\supset [0, \infty)$. We
proceed as in the previous section: to any $\mu\ge 0$ we are again
going to construct a sequence $\{\psi_n\}_{n=1}^\infty\subset
D(H_{\mathrm{Sm}} (A))$ of unit vectors, $\|\psi_n\|=1$, such that
$$ 
\|H_{\mathrm{Sm}} (A)\psi_n-\mu\psi_n\|\to 0 \qquad\text{as}\quad
n\to\infty\,.
$$
holds. The operator $L_0=-\frac{\mathrm{d}^2}{\mathrm{d}x^2}+\lambda
\delta(x)$ on $L^2(\mathbb{R})$ has a single eigenvalue equal to
$-\frac14\lambda^2$, hence its spectral threshold of is an isolated
eigenvalue corresponding to the normalized eigenfunction $h$, the
same as in the previous section. Given a smooth function $\chi$ with
$\mathrm{supp}\,\chi\subset[1, 2]$ and satisfying $\int_1^2
\chi^2(z)\, \mathrm{d}z=1$, we put
\begin{equation}
\label{sequence2} \psi_n(x,y):=h(x y)\,\e^{i \sqrt{\mu}
y}\chi\left(\frac{y}{n}\right)\,,
\end{equation}
where $n\in\mathbb{N}$ is to be chosen later. For the moment we just
note that choosing $n$ large enough one can achieve that
$\|\psi_n\|_{L^2(\mathbb{R}^2)}\ge\frac{1}{\sqrt{2}}$ as the
following estimates show,
\begin{eqnarray*}
\hspace{-4em}\int_{\mathbb{R}^2}\left|h(xy)\,\e^{i\sqrt{\mu}
y}\,\chi\left(\frac{y}{n}\right)\right|^2\,
\mathrm{d}x\,\mathrm{d}y= \int_{n}^{2n}\int_{\mathbb{R}}
\left|h(xy)\,\chi\left(\frac{y}{n}\right)
\right|^2\,\mathrm{d}x\,\mathrm{d}y \\[.2em] \hspace{-4em}
=\int_{n}^{2n}\int_{\mathbb{R}}\frac{1}{y}\left|h(t)\,\chi\left(\frac{y}{n}\right)
\right|^2\,\mathrm{d}t\,\mathrm{d}y
=\int_{n}^{2n}\frac{1}{y}\left|\chi\left(\frac{y}{n}\right)
\right|^2\,\mathrm{d}y=\int_1^2\frac{1}{z}\left|\chi(z)
\right|^2\,\mathrm{d}z\ge\frac{1}{2}\,.
\end{eqnarray*}
Next  we are going to show that $\|H_{\mathrm{Sm}} (A)\psi_n
-\mu\psi_n \|_{L^2(\mathbb{R}^2)}^2<\varepsilon$ holds for a
suitably chosen $n=n(\varepsilon)$. By a straightforward computation
we express $\frac{\partial^2\psi_n}{\partial x^2}$,
$\frac{\partial^2\psi_n}{\partial y^2}$, and $x \frac{\partial
\psi_n}{\partial y}$; in the same way as in the previous section one
can check that the norms of the last two can be made as small as we
wish by choosing $n$ sufficiently large. Moreover, we have
$$
\int_{\mathbb{R}^2}\left|x^2 h(x y)\,\e^{i \sqrt{\mu}
y}\chi\left(\frac{y}{n}\right)
\right|^2\,\mathrm{d}x\,\mathrm{d}y\le\frac{1}{n^4}
\int_1^2\frac{|\chi(z)|^2}{z}\mathrm{d}z\,\int_{\mathbb{R}}|h(t)|^2\,\mathrm{d}t\,,
$$
hence this term too can be made small. This allows us to estimate
the expression in question as
\begin{eqnarray*}
\hspace{-5em}\int_{\mathbb{R}^2}|H_{\mathrm{Sm}} (A)\psi_n-\mu
\psi_n|^2(x,y)\,\mathrm{d}x\,\mathrm{d}y \\
\hspace{-4em}=\int_{\mathbb{R}^2}\left|-\frac{\partial^2\psi_n}{\partial
x^2}- \frac{\partial^2\psi_n}{\partial y^2}+2i B x\frac{\partial
\psi_n}{\partial y}+B^2 x^2\psi_n+\omega^2y^2\psi_n-\mu
\psi_n\right|^2\, \mathrm{d}x\,\mathrm{d}y
\\\hspace{-4em}=\int_{n}^{2n}\int_{\mathbb{R}}\biggl|y^2\left(-h''(xy)
+\omega^2h(xy)\right)\chi\left(\frac{y}{n}\right)\biggr|^2
\,\mathrm{d}x\,\mathrm{d}y+\varepsilon
\end{eqnarray*}
for all sufficiently large $n$, and using the fact that $Lh=0$ holds
by assumption, we get from here
$$
\int_{\mathbb{R}^2}|H_{\mathrm{Sm}} (A)\psi_n-\mu\psi_n|^2(x,y)\,
\mathrm{d}x\,\mathrm{d}y<\varepsilon\,,
$$
which is what we have set out to demonstrate.
\end{proof}

\section{Spectrum of $H(A)$} \label{s:regul}
\setcounter{equation}{0}

Now we pass to the `regular' version of the magnetic
Smilansky-Solomyak model described by the Hamiltonian
\eqref{Hmagn2}. As before, the first question to address concerns
its self-adjointness. In this case we can check that $H(A)$ is
essentially self-adjoint on $C_0^\infty(\mathbb{R}^2)$ with a
reference to \cite{I90}: it is sufficient to find a sequence of
non-overlapping annular regions $A_m=\{z\in \mathbb{R}^2:\:
a_m<|z|<b_m\}$ and a sequence of positive numbers $\nu_m$ such that
\begin{equation}\label{ess}
(b_m-a_m)^2 \nu_m>K\,,\;\; V(z)\ge -k \nu^2_m
(b_m-a_m)^2\quad\text{for}\;\; z\in A_m \quad\text{and}\quad
\sum_{m=1}^\infty \nu_m^{-1}=\infty\,,
\end{equation}
where $K$ and $k$ are positive constants independent of $m$. It can
be seen easily that for $a_m=m,\,b_m=m+1$, and $\nu_m=m+1,\:m=0, 1,
2,\ldots$, the requirement (\ref{ess}) is satisfied if we choose
$K=\frac{1}{2}$ and $k=|\lambda| \| V\|_\infty$.

\subsection{Subcritical case: positivity and essential spectrum}
\label{ss:regul-posit}

As before we show first that the operator is positive in the
subcritical situation.
\begin{proposition}
$H (A)\ge0$ holds provided $\inf \sigma (L(V))\ge0$.
\end{proposition}
\begin{proof}
The argument is mimicking the reasoning used in
Proposition~\ref{subcr-posit}. For any $u\in \mathrm{Dom} (Q (H
(A)))\subset \mathcal{H}^1(\mathbb{R}^2)$ one has
\begin{eqnarray*}
\hspace{-5em} Q( H (A))(u)=\int_{\mathbb{R}^2}\left|i \frac{\partial
u}{\partial x}-B y u\right|^2+\left|\frac{\partial u}{\partial
y}\right|^2+\left(\omega^2 y^2+\lambda y^2 V(x y)\right)
|u|^2\,\mathrm{d}x\,\mathrm{d}y\\
\hspace{-4em} \ge\int_{\mathbb{R}^2}\left|i \frac{\partial
u}{\partial x}-B y u\right|^2+\left(\omega^2 y^2+\lambda y^2 V(x
y)\right) |u|^2\,\mathrm{d}x\,\mathrm{d}y\,.
\end{eqnarray*}
Furthermore, the quadratic forms
$$
u(\cdot,y) \mapsto \int_{\mathbb{R}}\left|i \frac{\partial
u}{\partial x}-B y u\right|^2+\left(\omega^2 y^2+\lambda y^2 V(x
y)\right) |u|^2\,\mathrm{d}x
$$
with a fixed $y$ correspond to the operators
$$
\left(i\frac{\mathrm{d}}{\mathrm{d} x}-B y\right)^2+\omega^2
y^2+\lambda y^2 V(x y)\quad\text{on}\quad
\mathcal{H}^1(\mathbb{R})\,,
$$
which are unitarily equivalent to $y^2 L(V)\ge0$.
\end{proof}

As in the case of $H_\mathrm{Sm}(A)$ the `unperturbed' essential
spectrum is preserved independently of the value the coupling
constant $\lambda$ may take.
\begin{theorem} \label{th:subcritess2}
$\sigma_{\mathrm{ess}} (H (A))\supset[\sqrt{\omega^2+B^2}, \infty)$.
\end{theorem}
\begin{proof}
As before we have to construct a Weyl sequence for any
$\mu\ge\sqrt{\omega^2+B^2}$, in other words, to find to any
$\varepsilon>0$ a function $\phi$  such that
\begin{equation}\label{Weyl ineq1.}
\|H (A) \phi-\mu \phi\|_{L^2 (\mathbb{R}^2)}<\varepsilon
\|\phi\|\,.
\end{equation}
We employ the functions defined by (\ref{Functions}) which obviously
belong to the domain of $H(A)$. The only change on the right-hand
side of (\ref{Weyl ineq.}) comes now from the addition of the term
$\lambda y^2 V(x y)\varphi_{k, \alpha, m}$. Using the fact that $V$
is by assumption compactly supported, we infer that
\begin{eqnarray*}
\hspace{-5em} \frac{1} {2\pi\, \mathrm{vol} (E)}
\int_{\mathbb{R}^2}y^2 \left|\int_E g\left(y-\frac{\xi
B}{\omega^2+B^2}\right) \,e^{i \xi (x-\alpha
k)}\,\mathrm{d}\xi\right|^2\,V^2(x y)\,
\eta^2\left(\frac{x}{k}\right)\,\chi^2\left(\frac{y}{k}\right)\,\mathrm{d}x\,\mathrm{d}y\\[.2em]
\hspace{-4em} \le\frac{\|\eta\|_\infty^2\,\|\chi\|^2_\infty
\|V\|_\infty^2}{\mathrm{vol} (E)}\int_{E\times\{|y|\le
\frac{s_0}{k}\}}y^2 g^2\left(y-\frac{x
B}{\omega^2+B^2}\right)\,\mathrm{d}x\,\mathrm{d}y \\[.2em]
\hspace{-4em} \le \frac{s_0^2 \|\eta\|_\infty^2\,\|\chi\|^2_\infty
\|V\|_\infty^2}{k^2 \mathrm{vol} (E)}\,\int_{E\times\mathbb{R}}
g^2\left(y-\frac{x
B}{\omega^2+B^2}\right)\,\mathrm{d}x\,\mathrm{d}y\\[.2em]
\hspace{-4em} =\frac{s_0^2 \|\eta\|_\infty^2\,\|\chi\|^2_\infty
\|V\|_\infty^2}{k^2}\,\int_{\mathbb{R}} g^2(z)\,\mathrm{d}z\,.
\end{eqnarray*}
Consequently, choosing $k$ large enough one can achieve that the
above integral will be sufficiently small, which together with the
inequality (\ref{final1}) implies the validity of \ref{Weyl ineq1.}.
The rest of the argument is the same as in
Theorem~\ref{th:subcritess}.
\end{proof}

\subsection{Subcritical case: essential spectrum threshold}
\label{ss:regul-discr}

While the value of $\lambda$ was irrelevant in the previous theorem, it becomes important if we ask about the essential spectrum threshold.
\begin{theorem} \label{reg-thresh}
Let $\inf\,\sigma(L(V))>0$, then the spectrum of $H (A)$ below $\sqrt{\omega^2+B^2}$
is purely discrete.
\end{theorem}
\begin{proof}
We employ Neumann bracketing combined with the minimax principle in a way similar to that used in \cite{BE17}. By $h^{(\pm)}_n (A, V)$ and $h_0 (A, V)$ we denote the Neumann restrictions of operator $H (A)$ to the strips
$$
G^{(\pm)}_n=\mathbb{R} \times \left\{y:\: 1+\ln n<\pm y\le
1+\ln(n+1)\right\},\,\,n\ge n_0\,,
$$
and $G_0=\mathbb{R}\times [-1-\ln n_0,1+\ln n_0]$, where $n_0\in\mathbb{N}$ will be chosen later. It allows us to estimated the operator from below,
\begin{equation}\label{N2}
H(A)\ge\left(\bigoplus_{n=n_0}^\infty\: h^{(+)}_n (A, V)\oplus
h^{(-)}_n (A, V)\right)\oplus h_0 (A, V)\,.
\end{equation}
To prove the result we have to show first that the spectral thresholds of $h_n^{(\pm)} (A, V)$ tend to infinity as $n\to\infty$, and secondly, that for any $\Lambda<\sqrt{\omega^2+B^2}$ one can choose $n_0$ in such a way that the spectrum of $h_0 (A, V)$ below
$\Lambda$ is purely discrete. The function $V$ is by assumption compactly supported with a bounded derivative, hence we have
$$
V(x y)-V(x (1+\ln n))=\mathcal{O}\left(\frac{1}{n\ln n}\right)\,,\quad y^2-(1+\ln n)^2=\mathcal{O}\left(\frac{\ln n}{n}\right)
$$
for any $(x,y) \in G_n^{(+)}$ and an analogous relation for $G^{(-)}_n$, and consequently,
\begin{equation}\label{ln}
y^2V(xy)-(1+\ln n)^2\,V(\pm x(1+\ln n))=\mathcal{O}\left(\frac{\ln n}{n}\right)
\end{equation}
holds for any $(x,y) \in G_n^{(\pm)}$. Moreover, for any function $u\in\mathcal{H}^1(G_n)$ and any fixed positive $\varepsilon$ we have
\begin{eqnarray*}
\hspace{-5em} \int_{G_n}\left|i \frac{\partial u}{\partial x}-B y u\right|^2\,\mathrm{d}x\,\mathrm{d}y
=\int_{G_n}\left|i \frac{\partial u}{\partial x}-B \ln n u+B(\ln n-y)u\right|^2\,\mathrm{d}x\,\mathrm{d}y\\
\hspace{-4em} \ge \int_{G_n}\left|i \frac{\partial u}{\partial x}-B \ln n u\right|^2\,\mathrm{d}x\,\mathrm{d}y
-2\int_{G_n}\sqrt{\varepsilon} \left|i \frac{\partial u}{\partial x}-B \ln n u\right| \frac{B}{\sqrt{\varepsilon}}|\ln n-y| |u|\,\mathrm{d}x\,\mathrm{d}y\\
\hspace{-4em} \ge(1-\varepsilon) \int_{G_n}\left|i \frac{\partial u}{\partial x}-B \ln n u\right|^2\,\mathrm{d}x\,\mathrm{d}y -\frac{B^2}{\varepsilon}\int_{G_n}(\ln (n+1)-\ln n)^2 |u|^2\,\mathrm{d}x\,\mathrm{d}y\,,
\end{eqnarray*}
and therefore
$$\left(i \frac{\partial}{\partial x}-B
y\right)^2\ge(1-\varepsilon) \left(i \frac{\partial}{\partial x}-B
\ln n\right)^2-\frac{B^2}{\varepsilon}\, (\ln (n+1)-\ln n)^2\,,
$$
which together with (\ref{ln}) implies the asymptotic inequalities
\begin{equation} \label{l_nk}
\inf\sigma(h_n^{(\pm)} (A, V))\ge(1-\varepsilon) \inf\sigma(l_n^{(\pm)} (A, V))+\mathcal{O}\left(\frac{\ln n}{n}\right)\,,
\end{equation}
in which the operators
$$
l_n^{(\pm)} (A, V):=\left(i\frac{\partial}{\partial x}-B \ln n\right)^2-\frac{\partial^2}{\partial y^2}+\omega^2(1+\ln
n)^2+\frac{\lambda}{1-\varepsilon} (1+\ln n)^2\,V(\pm x(1+\ln n))
$$
with Neumann conditions are defined on $G^{(\pm)}_n$.  Since $l_n^{\pm} (A, V)$  is easily seen to be unitarily equivalent to the non-magnetic  operator $\tilde{l}_n^{\pm}(V)=-\frac{\partial^2}{\partial x^2}-\frac{\partial^2}{\partial y^2}+\omega^2(1+\ln n)^2+\frac{\lambda}{1-\varepsilon} (1+\ln n)^2\,V(\pm x(1+\ln n))$ defined on the same domain, their spectra coincide. The operator $\tilde{l}_n^{\pm}(V)$ allows the separation of variables. Since the principal eigenvalue of $-\frac{\mathrm{d}^2}{\mathrm{d}y^2}$ on an
any interval with Neumann boundary conditions is zero, we have
\begin{equation}\label{lnk}
\inf \sigma (\tilde{l}_n^{\pm}(V))=\inf \sigma (l_n(V))\,,
\end{equation}
where $l_n(V)=-\frac{\mathrm{d}^2}{\mathrm{d}x^2}+\omega^2(1+\ln n)^2+\frac{\lambda}{1-\varepsilon} (1+\ln n)^2\,V(\pm x(1+\ln n))$. By the change of variable, $x=\frac{t}{1+\ln n}$, the last named operator is in turn unitarily equivalent to $(1+\ln n)^2 L_\varepsilon(V)$ with $L_\varepsilon(V) =\frac{\mathrm{d}^2}{\mathrm{d}t^2}+\omega^2+\frac{\lambda}{1-\varepsilon} V$, and therefore in view of inequality (\ref{lnk}) the relation $\inf \sigma (\tilde{l}_n^{\pm}(V))=(1+\ln n)^2 \inf \sigma (L_\varepsilon(V))$ holds, which together with the first order of perturbation-theory argument and (\ref{l_nk}) concludes the proof of the discreteness of $\bigoplus_{n=1}^\infty\:h^{(+)}_n (A, V)\oplus h^{(-)}_n (A, V)$.

It remains to inspect the spectrum of $h_0 (A, V)$. To proceed with the proof we need the following auxiliary result.
\begin{lemma}
Under our assumptions
\begin{equation}\label{stability2}
\inf \sigma_{\mathrm{ess}} (h_0 (A, V))=\inf \sigma_{\mathrm{ess}}(\tilde{h}_0 (A))\,,
\end{equation}
where $\tilde{h}_0 (A)$ is the operator $\left(i \frac{\partial}{\partial x}-By\right)^2-\frac{\partial^2}{\partial y^2}+\omega^2 y^2$ on
$L^2(G_0)$ with Neumann boundary conditions.
\end{lemma}
\begin{proof} From the minimax principle \cite[Secs.~XIII.1 and XIII.15]{RS78} it follows that
\begin{equation}\label{RS}
\inf\sigma_{\mathrm{ess}} (h_0 (A, V))\le\inf\sigma_{\mathrm{ess}}(\tilde{h}_0 (A))\,.
\end{equation}
To establish the opposite inequality it is enough to check that the spectrum of $h_0(A, V)$ is purely discrete below $\inf \sigma_{\mathrm{ess}} (\tilde{h}_0 (A))$. Given a $k\in\mathbb{N}$, we introduce the operator $h_1 (A,V)=\left(i \frac{\partial}{\partial x}-By\right)^2 -\frac{\partial^2}{\partial x^2}+\omega^2 y^2+\lambda y^2 V(x y)\chi_{\{|x|\le k\}} (x)$ for some large $k\in\mathbb{N}$. It differs from $h_0 (A, V)$ by the potential term in the region $\{|x|>k\}\times\mathbb{R}$, however, the potential $V$ is compactly supported by assumption, and therefore only $y\in\big(-\frac{s_0}{k}, \frac{s_0}{k}\big)$ must be considered and we get
$$
h_0(A, V)\ge h_1(A, V)-\frac{s_0^2 |\lambda| \|V\|_\infty}{k^2}\,,
$$
hence $\inf \sigma_{\mathrm{ess}} (h_0 (A, V))\ge\inf \sigma_{\mathrm{ess}} (h_1 (A, V))+\mathcal{O}\left(k^{-2}\right)$. Since $k$ can be chosen arbitrarily large, the identity \eqref{RS} would follow if we check that
\begin{equation}\label{stability3}
\sigma_{\mathrm{ess}} (h_1 (A,V))=\sigma_{\mathrm{ess}} (\tilde{h}_0 (A))\,.
\end{equation}
To this aim we use the stability of the essential spectrum against compact perturbations \cite[Sec.~XIII.4]{RS78}, specifically, we check the compactness of the resolvent difference $(h_1 (A, V)-z \mathbb{I})^{-1}-(\tilde{h}_0 (A)-z \mathbb{I})^{-1}$ as an operator on $L^2(G_0)$ for $z$ belonging to both the resolvent sets of $h_1 (A, V)$ and $\tilde{h}_0 (A)$. Using the resolvent identity we write the difference in question as
$$
(h_1 (A, V)-z \mathbb{I})^{-1} (h_1 (A, V)-\tilde{h}_0 (A)) (\tilde{h}_0 (A)-z \mathbb{I})^{-1}\,.
$$
It is easy to realize that for any bounded $\mathcal{U} \subset L^2(G_0)$ the set $(\tilde{h}_0 (A)-z \mathbb{I})^{-1} \mathcal{U}$ is uniformly $\mathcal{H}^1$ bounded. Furthermore, $h_1 (A, V)-\tilde{h}_0 (A)$ is by construction a compactly supported potential in $\{|x|\le k\}\times \{|y|\le n_0\}$, which implies its boundedness in $H^1(\Omega)$, where $\Omega=\{|x|\le k\}\times \{|y|\le n_0\})\cap\{(x, y): x y\in \mathrm{supp}\, V\}$. Finally, using the embedding theorems for Sobolev spaces on bounded domains we conclude that $(h_1 (A, V)-\tilde{h}_0 (A)) \mathcal{V}$ is compact in
$L^2(\Omega)$ which implies the same also for $(h_1 (A, V)-z \mathbb{I})^{-1} ((h_1 (A, V)-\tilde{h}_0 (A)) \mathcal{V})$, and thus the claim we have set out to prove.
\end{proof}

The rest is simple: combining the preceding lemma with the inclusion (\ref{sigmaess}) we verify the claim of Theorem~\ref{reg-thresh}.
\end{proof}

\subsection{Subcritical case: existence of the discrete spectrum}
\label{ss:regul-existence}

As in the previous section, proving that the spectrum below $\sqrt{\omega^2+B^2})$ says nothing about its existence, it has to be checked separately.
\begin{theorem}
Let $\inf \sigma (L(V))>0$, then the discrete spectrum of $H (A)$ is
non-empty and contained in the interval $(0, \sqrt{\omega^2+B^2})$.
\end{theorem}
\begin{proof}
We have to construct a normalized trial function $\phi$ such that the corresponding value of the quadratic form $Q (H(A))$ will be less than
$\sqrt{\omega^2+B^2}$. This time we employ the letter $h$ to denote the normalized ground-state eigenfunction of the one-dimensional harmonic oscillator, $h_{\mathrm{osc}}=-\frac{\mathrm{d}^2}{\mathrm{d}y^2}+(\omega^2+B^2)y^2$ on $L^2(\mathbb{R})$, and set
$$
\phi(x, y):=\frac{1}{\sqrt{k}} h(y)\chi\left(\frac{x}{k}\right)\,,
$$
where $\chi(z)$ is a real-valued smooth function with $\mathrm{supp} (\chi)=[-1, 1]$ such that
$$
\int_{-1}^1\chi^2(z)\,\mathrm{d}z=1\,,\quad\min_{|z|\le1/2}\,
\chi(z)=:\alpha>0\,,
$$
and $k$ is a natural number to be chosen later. A straightforward computation yields
\begin{align*}
Q (H(A))[\phi]= & \int_{\mathbb{R}^2}\left|\frac{\partial{\phi}}{\partial{x}}\right|^2\,
\mathrm{d}x\,\mathrm{d}y
+\int_{\mathbb{R}^2}\left|\frac{\partial{\phi}}{\partial{y}}\right|^2\,
\mathrm{d}x\,\mathrm{d}y +\int_{\mathbb{R}^2}(\omega^2+B^2) y^2\,
|\phi|^2\,\mathrm{d}x\,\mathrm{d}y \\ & + \lambda
\int_{\mathbb{R}^2}y^2 V(x y)\,|\phi^2|\,\mathrm{d}x\,\mathrm{d}y\,,
\end{align*}
because the contribution from the terms containing the first derivatives is easily seen to vanish, and therefore
\begin{eqnarray}
\nonumber
\hspace{-5em} Q (H(A))[\phi]
=\frac{1}{k^3}\int_{\mathbb{R}^2}h^2(y)\,(\chi^\prime)^2
\left(\frac{x}{k}\right)\,\mathrm{d}x\,\mathrm{d}y
+\frac{1}{k}\int_{\mathbb{R}^2}\left(h^\prime\right)^2(y)\,
\chi^2\left(\frac{x}{k}\right)\,\mathrm{d}x\,\mathrm{d}y
\\ \nonumber \hspace{-2em} +\frac{1}{k}\int_{\mathbb{R}^2}(\omega^2+B^2) y^2
\,h^2(y)\,\chi^2\left(\frac{x}{k}\right)\,\mathrm{d}x\,\mathrm{d}y
+\frac{\lambda}{k}\int_{\mathbb{R}^2}y^2 V(x y)\, h^2(y)\,\chi^2
\left(\frac{x}{k}\right)\,\mathrm{d}x\,\mathrm{d}y
\\ \nonumber \hspace{-5em}
=\mathcal{O}(k^{-2})+\frac{1}{k}
\int_{\mathbb{R}^2}\left(\left(h^\prime\right)^2(y) +(\omega^2+B^2)
y^2\,h^2(y)\right)\,\chi^2\left(\frac{x}{k}\right)\,
\mathrm{d}x\,\mathrm{d}y
\\ \nonumber \hspace{-2em} +\frac{\lambda}{k}\int_{\mathbb{R}^2}y^2
V(x y)\, h^2(y)\, \chi^2\left(\frac{x}{k}\right)\,
\mathrm{d}x\,\mathrm{d}y
\\ \nonumber \hspace{-5em}
=\mathcal{O}(k^{-2})+\frac{\sqrt{\omega^2+B^2}}{k}\int_{\mathbb{R}^2}
h^2(y)\,\chi^2\left(\frac{x}{k}\right)\,\mathrm{d}x\,\mathrm{d}y
+\frac{\lambda}{k}\int_{\mathbb{R}^2}y^2 V(xy)\, h^2(y)\,
\chi^2\left(\frac{x}{k}\ \right)\,\mathrm{d}x\,\mathrm{d}y
\\ \label{non-emptiness} \hspace{-5em}
=\mathcal{O}(k^{-2})+\sqrt{\omega^2+B^2}+\frac{\lambda}{k}\int_{\mathbb{R}^2}y^2
V(x y)\, h^2(y)\,
\chi^2\left(\frac{x}{k}\right)\,\mathrm{d}x\,\mathrm{d}y\,.
\end{eqnarray}
We need to estimate the last term on the right-hand side of
(\ref{non-emptiness}). One has
\begin{eqnarray*}
\hspace{-5.5em} \lefteqn{\frac{|\lambda|}{k}\int_{\mathbb{R}^2}y^2 V(x y)\, h^2(y)\,
\chi^2\left(\frac{x}{k}\right)\,
\mathrm{d}x\,\mathrm{d}y=\frac{|\lambda|}{k}\int_{-k}^k
\int_{\mathbb{R}} y^2 V(x y)\, h^2(y)\,
\chi^2\left(\frac{x}{k}\right)\,\mathrm{d}x\,\mathrm{d}y} \\ &&
\hspace{-5em} \ge\frac{|\lambda|}{k}\int_{-k/2}^{k/2} \int_0^\infty y^2 V(x y)\,
h^2(y)\,
\chi^2\left(\frac{x}{k}\right)\,\mathrm{d}x\,\mathrm{d}y\ge\frac{\alpha^2
|\lambda|}{k}\int_{-k/2}^{k/2} \int_0^\infty y^2 V(x y)\,
h^2(y)\,\mathrm{d}x\,\mathrm{d}y\\ && \hspace{-5em} =\frac{\alpha^2
\lambda}{k}\int_0^\infty\int_{-k y/2}^{k y/2}y\, V(t)\,
h^2(y)\,\mathrm{d}t\,\mathrm{d}y\ge\frac{\alpha^2
|\lambda|}{k}\int_1^\infty\int_{-k/2}^{k/2}y V(t)\,
h^2(y)\,\mathrm{d}t\,\mathrm{d}y\\ && \hspace{-5em} \ge\frac{\alpha^2
|\lambda|}{k}\int_1^\infty y
h^2(y)\,\mathrm{d}y\,\int_{-k/2}^{k/2}V(t)\,\mathrm{d}t\,.
\end{eqnarray*}
If $k$ is chosen large enough the above estimate implies
$$
\frac{|\lambda|}{k}\int_{\mathbb{R}^2}y^2 V(x y)\, h^2(y)\,
\chi^2\left(\frac{x}{k}\right)\,\mathrm{d}x\,\mathrm{d}y\ge\frac{\alpha^2
|\lambda|}{k}\int_1^\infty y
h^2(y)\,\mathrm{d}y\,\int_{-s_0}^{s_0}V(t)\,\mathrm{d}t\,,
$$
hence in combination with (\ref{non-emptiness}) we infer that
$$
Q (H (A))[\phi] \le
\mathcal{O}(k^{-2})+\sqrt{\omega^2+B^2}-\frac{\alpha^2
|\lambda|}{k}\int_1^\infty y
h^2(y)\,\mathrm{d}y\,\int_{-a}^aV(t)\,\mathrm{d}t\,,
$$
where the right-hand side is obviously less than $\sqrt{\omega^2+B^2}$ for all $k$ large enough.
\end{proof}

\subsection{The supercritical and critical cases}
\label{ss:regul-rest}

Let us turn next to the case where the `escape to infinity' is possible.
\begin{theorem}
Under our hypotheses, assuming in addition that the potential $V$ is symmetric with respect to the origin, $\sigma(H(A))=\mathbb{R}$ holds if
$\inf\sigma(L(V))<0$.
\end{theorem}
\noindent The proof is completely the same as in Theorem~\ref{th:supercrit}, the only difference concerns the substitution of the function $h$ into the normalized eigenfunction of $L(V)$. The symmetry of the potential $V$ is needed to guarantee the existence of the solutions of the differential
equation (\ref{diff.eq.}).

\bigskip

Finally, in the critical case we have
\begin{theorem}\label{th:critess2}
$\sigma (H(A))=\sigma_{\mathrm{ess}} (H(A))=[0, \infty)$ holds provided $\inf\sigma(L(V))=0$.
\end{theorem}
\noindent 
The proof repeats the almost exactly the argument leading to Theorem~\ref{th:critess}.

\subsection*{Acknowledgements}
The authors are grateful to V.~Lotoreichik for a useful discussion.
The research has been supported by the Czech Science Foundation
(GA\v{C}R) within the project 17-01706S. D.B. also acknowledges a
support from the projects 01211/2016/RRC and the Czech-Austrian
Grant CZ 02/2017.


\end{document}